\documentclass[a4paper,11pt]{amsart}

\title{Space-time fractional NLS equation}
\author{Ricardo Grande}
\address{Massachusetts Institute of Technology, Department of Mathematics}
\email{rgi@mit.edu}
    
\usepackage{amsfonts,mathrsfs,amssymb,amsthm}
\usepackage{url}
\newcommand{\norm}[1]{\left\lVert#1\right\rVert}

\usepackage{mathtools}

\def \RR{\mathbb R}
\def \CC{\mathbb C}

\allowdisplaybreaks

\usepackage[hidelinks]{hyperref}
\usepackage{cleveref}
\usepackage{amsrefs}

\numberwithin{equation}{section}

\theoremstyle{plain}
\newtheorem{theorem}{Theorem}[section]
\newtheorem{proposition}{Proposition}[section]
\newtheorem{lemma}{Lemma}[section]

\theoremstyle{remark}
\newtheorem{rem}{Remark}[section]

\begin{document}

\begin{abstract}
 In this paper we prove local well-posedness of a space-time fractional generalization of the nonlinear Schr\"odinger equation with a power-type nonlinearity. The linear equation coincides with a model proposed by Naber, and displays a nonlocal behavior both in space and time which accounts for long-range interactions and a so-called memory effect. Because of a loss of derivatives produced by the latter and the lack of semigroup structure of the solution operator, we employ a strategy of proof based on exploiting some smoothing effect similar to that used by Kenig, Ponce and Vega for the KdV equation. Finally, we prove analytic ill-posedness of the data-to-solution map in the supercritical case.
\end{abstract}

\maketitle

\section{Introduction}

\subsection{Space-time fractional NLS}

In this article we study the initial value problem 
\begin{equation}\label{linear}
\left\{ \begin{array}{ll}
i^{\beta} \partial_t^{\beta} u & = (-\Delta_x)^{\alpha /2}\, u + g(u) \qquad (t,x)\in (0,\infty)\times\RR , \\
u\mid_{t=0} & = f \in H^s (\RR).
\end{array}\right.
\end{equation}
for $0<\alpha <2$ and $0<\beta <1$. 
We consider nonlinearities of polynomial type $g(u)=\mu |u|^{p-1} u$ for an odd integer $p$ and $\mu=\pm 1$. The operator $(-\Delta_x)^{\alpha /2}$, known as the fractional Laplacian, is given by the Fourier multiplier of symbol $|\xi |^{\alpha}$, and the operator with symbol $\partial_t^{\beta}$ is the Caputo fractional derivative, given by 
\[\partial^{\beta}_t u (t,x)=\frac{1}{\Gamma (1-\beta)} \int_0^{t} \frac{\partial_{\tau}u(\tau,x)}{(t-\tau)^{\beta}} \, d\tau .\]

The case $\alpha=2$ and $\beta=1$ is simply the nonlinear Schr\"odinger equation, which has been extensively studied, see the books \cite{caze} and \cite{tao} for some results and a list of contributing authors.

Laskin proposed equation \cref{linear} in the case $\beta=1$ (i.e. classical derivative in time) as a fundamental equation in fractional quantum mechanics, explaining that ``if the [Feynmann] path integral over Brownian trajectories leads to the well known Schr\"odinger equation, then the path integral over L\'evy trajectories leads to the fractional Schr\"odinger equation'' \cite{Laskin}. Later, Naber proposed the time-fractional case arguing that the process could be further generalized to be non-Markovian at all, thus giving rise to a memory effect \cite{Naber}.

As explained in \cite{prob}, if a Cauchy problem $\partial_t u=Au$ admits a stochastic solution $X(t)$, in the sense that the probability density of such a stochastic process solve the Cauchy problem, then under some technical conditions, $\partial^{\beta}_t u=Au$ admits $X(E_t)$ as a solution, where $E_t = \inf\{ x > 0 \mid D_x > t \}$ is the inverse hitting time of a L\'evy process $D_x$ whose probability density function has $e^{-s^{\beta}}$ as its Laplace transform. In fact, stochastic solutions to some fractional Cauchy problems arise as scaling limits of continuous time random walks whose i.i.d. jumps are separated by i.i.d. waiting times, where the probability of waiting longer than time $t>0$ decays like $t^{-\beta}$ for large $t$, see \cite{prob}.

The one-dimensional space-fractional cubic equation \cref{linear}, with $\beta=1$ and $1<\alpha<2$, was first studied in \cite{cubic}. This work was later generalized to every dimension and any power-type nonlinearity in \cite{HS}, where Hong and Sire developed a general local and global well-posedness theory for the space-fractional equation provided the regularity $s$ is greater or equal than the regularity invariant under the Galilean transformation, $s_g=\frac{1}{2}-\frac{\alpha}{4}$, and the regularity invariant under scaling, $s_c=\frac{1}{2}-\frac{\alpha}{p-1}$. Their work was further extended in \cite{toulouse} (see also a series of papers about blow-up \cite{toulouse2,toulouse3}).

The combined space-time fractional linear equation with a potential has also been studied recently, see for instance \cite{dx}. A different type of linear space-time fractional equation was proposed and studied in \cite{french}. However one should note that their definition of the nonlocal time derivative, although allowing them to keep some group property on the solution operator, does not agree with what Naber proposed in \cite{Naber}. 

Instead we propose \cref{linear} as a generalization of the nonlinear Schr\"odinger equation whose linear part agrees with Naber's work.
The coefficient $i^{\beta}$, as opposed to $i$, has been a matter of discussion. Naber argues in favor of $i^{\beta}$, and among other reasons, he explains that after taking the Laplace transform in time and Fourier transform in space, the choice of $i^{\beta}$ produces a movement of the pole of the solution along the imaginary axis as $\beta$ ranges between 0 and 1. However, if one instead chooses $i$, the pole would move to almost any desired location in the complex plane. Physically, this would mean that a small change in the order of the time derivative could change the temporal behavior from sinusoidal to growth or to decay. Moreover, we would add that if one explores the possibility of extending this work allowing exponents $\beta$ to vary in the range $(1,2)$, the choice of $i^{\beta}$ provides equations which interpolate between the Schr\"odinger and wave equations, allowing one to remain within the domain of dispersive equations.

Recently, the case of the coefficient $i$ instead of $i^{\beta}$ has been studied in \cite{c2,c1}. See \cref{cpapers} for a more detailed discussion about this work.

\subsection{Background about fractional equations}

Classical equations such as the Laplace, heat, wave and even Schr\"odinger equations admit a generalization with a fractional Laplacian replacing the Laplacian, most commonly defined as a Fourier multiplier:
\[\widehat{(-\Delta)^sf}=|\xi |^{2s} \widehat{f}(\xi ) \  \mbox{for}\ s >0.\]
We will refer to such equations as space-fractional equations. 

The case of the space-fractional Laplace equation is well-known, see for instance Stein's book \cite{Stein} for an approach based on the use of Riesz potentials, or \cite{Nicola} for a recent survey. The case of the space-fractional porous medium equation, with a fractional Laplacian and a classical time derivative, has also been studied in a series of papers \cite{PME1,PME2}.

A different approach in the generalization of such classical equations is the substitution of the time derivative for its fractional counterpart. This is a concept that mainly comes from applications, as there are evolution processes whose modelling requires to take into account the past, thus exhibiting a nonlocal behavior in time. There are different ways of making sense of such a fractional derivative in time, but one of the most common ones is to use the Caputo derivative:
\[\partial^{\beta}_t f (t)=\frac{1}{\Gamma (1-\beta)} \int_0^{t} \frac{f'(\tau)}{(t-\tau)^{\beta}} \, d\tau ,\]
for $0<\beta<1$, where $\Gamma$ is the Gamma function.  

Such an approach has been taken by Allen, Caffarelli and Vasseur when considering a space-time fractional porous medium equation \cite{fracPME1,fracPME2}.
These types of problems involving fractional time derivatives have also been studied using probabilistic techniques, for example in \cite{prob}, where the authors develop stochastic solutions to Cauchy problems of type $\partial^{\beta}_t u=A u$ on bounded domains via solutions to the classical Cauchy problem $\partial_t u=A u$.

Fractional derivatives have been used to model some phenomena in Physics, such as non-diffusive transport in  plasma  turbulence \cite{Negrete1}, and in Economics, such as ruin theory of insurance companies, growth and inequality processes, and high-frequency price fluctuation in financial markets \cite{econ}.

\subsection{Background and properties of fractional NLS}

Consider the linear case of equation \cref{linear}, i.e. $g=0$. One may take the Laplace transform in time, the Fourier transform in space, and solve the resulting equation to formally find 
\begin{equation}\label{Lsolution}
u(t,x)=\int_{\RR} e^{ix\cdot \xi} \, \widehat{f}(\xi) \, \left[ \sum_{k=0}^{\infty} \frac{ t^{\beta\, k} |\xi|^{\alpha\, k} i^{- \beta\, k}}{\Gamma(\beta\, k +1)}\right] \, d\xi .
\end{equation}
The power series comes from the {\it Mittag-Leffler function} $E_{\beta}(|\xi |^{\alpha} t^{\beta} i^{-\beta})$, where
\begin{equation}\label{ML}
E_{\beta}(z)=\sum_{k=0}^{\infty} \frac{z^{k}}{\Gamma(\beta k +1)} ,
\end{equation}
which is an entire function in the complex plane. More details about this derivation can be found in \cite{kai} and \cref{sec: appA}.

We can also perform some scaling analysis for equation \cref{linear}. If $u$ is a solution to the equation then so is
\[u_{\lambda}(t,x)=\lambda^{\frac{\alpha\beta}{p-1}}\, u(\lambda^{\alpha} t,\lambda^{\beta} x) ,\]
with the obvious rescaling of the initial data. Then one can quickly check that the critical regularity invariant under scaling is
\[s_c=\frac{1}{2}-\frac{\alpha}{p-1}\ .\]
Note that $\beta$ plays no role on this formula, and so it coincides with the space-fractional case treated in \cite{HS} for dimension 1.

An interesting feature is that even the linear equation has no conserved quantities. However, one can use the following asymptotic expansion for the Mittag-Leffler function, which may be found in Chapter 18 of \cite{bate},
\begin{equation}\label{Lasymp}
E_{\beta}(z) = \frac{1}{\beta} \exp (z^{\frac{1}{\beta}}) - \sum_{k=1}^{N-1} \frac{z^{-k}}{\Gamma (1-\beta k)}+ O(|z|^{-N})\qquad \mbox{as}\ |z|\rightarrow\infty ,
\end{equation}
to control $\norm{u(t)}_{L^2_x} \leq C$ uniformly in time. In fact, one can even show that 
\[\lim_{t\rightarrow\infty} \norm{u(t)}_{L^2_x}=\frac{1}{\beta}\norm{f}_{L^2_x} ,\]
i.e. as time passes the mass grows towards $\frac{1}{\beta}$ times that of the initial data. 
Formula \cref{Lasymp} is valid when $|\arg (z)|\leq \frac{\pi}{2}\beta$ and for any integer $N\geq 2$.
In fact, we will always choose the branch of the complex logarithm for which $|\arg (z) |<\pi$, so that $i^{-\beta}=e^{-i\beta\frac{\pi}{2}}$ and $i^{\beta}=e^{i\beta\frac{\pi}{2}}$.

By using a fractional generalization of the Duhamel formula, we can write the solution of the nonlinear problem \cref{linear} as 
\begin{align}\label{inteq}
u(t,x) = &\int_{\RR} e^{ix\cdot \xi} \, \widehat{f}(\xi) \, E_{\beta}(|\xi |^{\alpha} t^{\beta} i^{-\beta}) \, d\xi \\
& + i^{-\beta} \, \int_{\RR} \int_0^{t} \widehat{g}(\tau , \xi) \, (t-\tau )^{\beta -1} E_{\beta , \beta }(i^{-\beta} (t-\tau )^{\beta} |\xi |^{\alpha}) \, e^{ix\cdot \xi}\, d\tau d\xi  ,\nonumber
\end{align}
where 
\[\widehat{g}(t,\xi)= \int_{\RR} g(u(t,x)) \, e^{-i \xi \cdot x} \, dx ,\]
and 
\begin{equation}\label{genML}
E_{\beta , \beta}(z)= \sum_{k=0}^{\infty} \frac{z^k}{\Gamma (\beta k +\beta )}.
\end{equation}
This is known as the {\it generalized Mittag-Leffler function}, and is an entire function. As before, more details about this can be found both in \cref{sec: appA} and \cite{kai}.

Because of the difficulties in dealing with this function directly, the following asymptotic formula is helpful, which may also be found in Chapter 18 of \cite{bate},
\begin{align}\label{NLasymp}
t^{\beta -1} E_{\beta, \beta}(i^{-\beta} t^{\beta} |\xi |^{\alpha}) = & \frac{1}{\beta} i^{\beta -1} |\xi |^{\sigma - \alpha} \, e^{-i t |\xi |^{\sigma}} \\
& - \sum_{k=2}^{N} \frac{\Gamma (\beta -\beta k )^{-1} \, i^{\beta k}}{t^{1 + \beta (k-1)} |\xi |^{\alpha k}}+O \left(\frac{1}{t^{1 + \beta N} |\xi |^{\alpha (N+1)}}\right)\nonumber
\end{align}
and is valid as $t |\xi |^{\sigma} \rightarrow \infty$, and for any $N\geq 2$, where $\sigma=\frac{\alpha}{\beta}$. We again observe an oscillatory part and a monotonous part, which will need to be managed separately.

\begin{rem}
As seen in \cref{Lsolution,NLasymp}, the solution operator does not enjoy the usual group or even semigroup property with respect to time. This is a major obstacle that prevents us from using Strichartz estimates and other techniques that normally rely on this fact. As a result, our treatment is substantially different from the techniques used both for classical NLS and for the purely space-fractional equation in \cite{HS}.
\end{rem}

\subsection{Statement of results}

Using the notation presented in Subsection 1.5, we now present a local well-posedness result of equation \cref{linear} for a special choice of parameters.

\begin{theorem}\label{maincor} Consider the space-time fractional nonlinear Schr\"odinger initial value problem:
\begin{equation}\label{parteq}
\left\{ \begin{array}{ll}
i^{\frac{7}{8}} \partial_t^{\frac{7}{8}} u & = (-\Delta_x)^{\frac{7}{8}}\, u + \mu |u|^2 u \qquad (t,x)\in (0,\infty )\times\RR , \\
u\mid_{t=0} & = f ,
\end{array}\right.
\end{equation}
Then for every $f\in H^{\frac{1}{4}}(\RR)$ there exists $T=T(\norm{f}_{H^{1/4}(\RR)})>0$ (with $T(\rho )\rightarrow\infty$ as $\rho\rightarrow 0$) and a unique solution $u(t,x)$ to the associated integral equation given by \cref{inteq} satisfying
\begin{equation}\label{partclass1}
u\in C([0,T],H^{\frac{1}{4}} (\RR)),
\end{equation}
\begin{equation}\label{partclass2}
\norm{\langle \nabla \rangle^{\frac{3}{4}-} u }_{L^{\infty}_x L^2_T} < \infty ,
\end{equation}
and 
\begin{equation}\label{partclass3}
\norm{u}_{L^{4}_x L^{\infty}_T} <\infty .
\end{equation}
Moreover, for any $T'\in (0,T)$ there exists a neighborhood $V$ of $f$ in $H^{\frac{1}{4}} (\RR )$ such that the map $\tilde{f} \rightarrow \tilde{u}$ from $V$ into the class defined by \cref{partclass1,partclass2,partclass3} with $T'$ instead of $T$ is Lipschitz.
\end{theorem}

The techniques used in the proof of \cref{maincor} can be modified to prove the following more general result.

\begin{theorem}\label{maintheorem} Consider the space-time fractional nonlinear Schr\"odinger initial value problem \cref{linear} with a nonlinearity $g(u)=\mu |u|^{p-1} u$ for some odd integer $p\geq 3$, $\mu=\pm 1$,  $\alpha>0$ and $\beta\in (0,1)$. With $\sigma=\frac{\alpha}{\beta}$, suppose that 
\begin{equation}\label{parameterconditions}
\alpha > \frac{\sigma +1}{2}, \quad s\geq \frac{1}{2}-\frac{1}{2(p-1)},\quad \mbox{and}\quad \delta\in \left[s+\sigma-\alpha, \frac{\sigma}{2}-\frac{1}{2(p-1)}\right).
\end{equation}
for some $s\in\RR$. Then for every $f\in H^s(\RR)$ there exists $T=T(\norm{f}_{H^s(\RR)})>0$ (with $T(\rho )\rightarrow\infty$ as $\rho\rightarrow 0$) and a unique solution $u(t,x)$ to the integral equation given by \cref{inteq} satisfying
\begin{equation}\label{class1}
u\in C([0,T],H^s (\RR)),
\end{equation}
\begin{equation}\label{class2}
\norm{\langle \nabla \rangle^{\delta} u }_{L^{\infty}_x L^2_T} < \infty ,
\end{equation}
and 
\begin{equation}\label{class3}
\norm{u}_{L^{2(p-1)}_x L^{\infty}_T} <\infty .
\end{equation}
Moreover, for any $T'\in (0,T)$ there exists a neighborhood $V$ of $f$ in $H^s (\RR )$ such that the map $\tilde{f} \rightarrow \tilde{u}$ from $V$ into the class defined by \cref{class1,class2,class3} with $T'$ instead of $T$ is Lipschitz. 
\end{theorem}

Finally, we present a result of analytic ill-posedness for supercritical regularity. According to Holmer \cite{JH}, this method was introduced by Bourgain in \cite{Bourgain}. This type of result appears for example in the work of Holmer himself for the 1D Zhakarov system, \cite{JH}, the work of Germain on the Navier-Stokes equation, \cite{Germain}, as well as the results of Molinet, Saut and Tzvetkov, \cite{molinet}, in the case of the KP-I equation. Lebeau was the first to show the stronger result of loss of regularity in his work on the nonlinear wave equation \cite{lebeau1,lebeau2}.
Regarding the nonlinear Schr\"odinger equation, Bejenaru and Tao proved ill-posedness of the quadratic Schr\"odinger equation, \cite{bejenaru}, Thomann proved loss of regularity for NLS on analytic Riemannian manifolds, \cite{thomann}, and Alazard and Carles
showed loss of regularity for supercritical NLS, \cite{alazard}, among other results.

\begin{theorem}\label{illtheorem} Consider the initial value problem:
\[\left\{
    \begin{array}{ll}
    i^{\beta}\partial_t^{\beta} u & = (-\Delta_x)^{\frac{\alpha}{2}}u + \mu|u|^{p-1} u,\quad (t,x)\in [0,T]\times\RR,\\
    u|_{t=0} & = f\in H^s(\RR) ,
    \end{array}
\right.\]
where $\alpha\in (0,2)$, $\beta\in (0,1)$, and $\mu=\pm 1$. If $p\geq 3$ is an integer, then the initial data-to-solution map from $H^s(\RR)$ to $C_t([0,T],H^s_x(\RR))$ is not $C^p$ for $s<s_c=\frac{1}{2}-\frac{\alpha}{p-1}$. If $p=2$, then the initial data-to-solution map from $H^s(\RR)$ to $C_t([0,T],H^s_x(\RR))$ is not $C^2$ for any $s$.
\end{theorem}

We finish this section with a few remarks about the results in this paper.

\begin{rem}\label{cpapers} Between the publication of the first version of this article on the Arxiv and the current one, two papers were posted about a similar equation with coefficient $i$ instead of $i^{\beta}$ \cite{c2,c1}:
\begin{equation}\label{nopower}
\left\{
    \begin{array}{ll}
    i\partial_t^{\beta} u & = (-\Delta_x)^{\frac{\alpha}{2}}u + \mu |u|^{p-1}u, \quad (t,x)\in (0,T)\times \RR^n,\\
    u|_{t=0} & =f .
    \end{array}
\right.
\end{equation}
For initial data $f\in L^r(\RR^n)$, for general dimension $n$ and under certain technical conditions on the admissible triplet $(q,r,p)$, the authors prove local well-posedness of \cref{nopower} in the space $C_b \left([0,T),L^r(\RR^n)\right)\cap L^q\left([0,T),L^p(\RR^n)\right)$.

Beyond our arguments in favor of $i^{\beta}$ instead of $i$ explained in the introduction, one might wonder whether similar techniques could be used in our equation. Unfortunately, the Fourier multiplier that appears in the solution to \cref{nopower}, $E_{\beta}(-i t^{\beta}|\xi|^{\alpha})$, displays very different asymptotics compared to ours, $E_{\beta}(i^{-\beta} t^{\beta}|\xi|^{\alpha})$. As an illustrative example, consider the case $\beta\in (\frac{1}{2},1)$, where \cref{Lasymp} is still valid. Note that the argument in the exponential has negative real part for $z=-i t^{\beta}|\xi|^{\alpha}$, and thus decays. Therefore their asymptotics are in fact dominated by the ``good'' term $t^{-\beta}|\xi|^{-\alpha}$ and the mass of their solution tends to zero as time passes. The same happens with the Fourier multiplier acting on the nonlinearity, which is now $E_{\beta,\beta}(-it^{\beta}|\xi|^{\alpha})$. Its leading behavior is now given by a good term, instead of a complex exponential as in \cref{NLasymp}, and so there is no loss of derivatives that needs to be overcome. This allows for the use of standard techniques similar to Strichartz estimates that lead to the results explained above. However, one could argue that the fact that their Fourier multiplier does not display an asymptotically oscillatory behavior (and instead decays) is an argument against the use of \cref{nopower} as a generalization of classical Schr\"odinger for fractional $\beta$.
\end{rem}

\begin{rem} Note that condition \cref{parameterconditions} in \cref{maintheorem} together with $\alpha>0$ and $\beta\in (0,1)$ imply that $\sigma>\alpha>1$ and $\beta>\frac{1}{2}$.
\end{rem}

\begin{rem} It is possible to generalize \cref{maintheorem} for even and non-integer $p$, but the exposition becomes more intricate (and slightly restricting the range of parameters might be necessary). The main idea is to use the estimates in this paper with respect to the norms \cref{class2,class3}, together with interpolation theorems to develop additional linear estimates. This strategy is based on the work of Kenig, Ponce and Vega for the KdV equation in \cite{KPV}. One will need the fractional chain rule instead of  \cref{AvoidChainRule}, which can be found in \cite{KPV}, too.
\end{rem}

\begin{rem} Note that there is a gap in the regularities where we have local well-posedness, given by \cref{maintheorem}, and the ill-posedness result of \cref{illtheorem}. For $p\geq 5$, and $s$ in the range
\[ \frac{1}{2}-\frac{1}{2(p-1)}>s\geq s_c=\frac{1}{2}-\frac{\alpha}{p-1},\]
local well-posedness remains an open question.

In the case $p=3$, we probably should not expect local well-posedness all the way to $s_c$. Indeed, in the case $\beta=1$, local well-posedness is obtained when $s>\max\{s_c,s_g\}$, where 
\[s_g:=\frac{1}{2}-\frac{\alpha}{4},\]
is the regularity that is invariant under the (pseudo) Galilean invariance, see \cite{HS}. When $p\in (1,5)$ and $\beta=1$, local well-posedness holds only for $s>\frac{1}{2}-\frac{\alpha}{4}>s_c$, and this result is sharp in the cubic case, as proved in \cite{cubic}. This suggests that a similar thing might happen for \cref{linear} when $p=3$.

In the case of the KdV equation, extending the local well-posedness results presented in \cite{KPV} to negative regularities required the use of $X^{s,b}$ spaces, see \cite{KPV2}. Similar ideas might be necessary to overcome the gap in our case too.

Incidentally, closing this gap would probably allow an accurate comparison between the solution to \cref{linear} and the second iterate described in \cref{illtheorem}, in order to prove a stronger result of ill-posedness known as norm inflation.
\end{rem}

\begin{rem} Because of \cref{NLasymp}, if one tries to take the $L^{\infty}_T L^2_x$ norm of 
\[\int_{\RR} \int_0^{t} \widehat{g}(\tau , \xi) \, (t-\tau )^{\beta -1} E_{\beta , \beta }(i^{-\beta} (t-\tau )^{\beta} |\xi |^{\alpha}) \, e^{ix\cdot \xi}\, d\tau d\xi ,\]
we seem to lose $\sigma-\alpha$ derivatives, which is an obstacle to closing the contraction-mapping argument. 
In order to circumvent this problem, we exploit some smoothing effect of the linear operator, which explains the choice of norm in \cref{class2}.
Note that this is not an issue in the space-fractional case ($\beta=1$) because then $\sigma=\alpha$.
\end{rem}

\begin{rem} The smoothing effect mentioned above is however limited. As will be seen in \cref{smoothingprop}, one can balance the loss of derivatives only if 
\begin{equation}\label{smoothingcondition}
\frac{\sigma-1}{2}>\sigma-\alpha,
\end{equation} which restricts the range of parameters to that presented in \cref{parameterconditions}. In fact, note that \cref{smoothingcondition} directly implies that $\alpha >1$ and $\beta >\frac{1}{2}$, among other things.
\end{rem}

\begin{rem} If one wants to generalize this theorem to a higher dimension $n$, a good place to start is the generalization of \cref{smoothingprop}, which is straight-forward. However, condition \cref{smoothingcondition}, necessary to overcome the loss of derivatives in the nonlinear term, becomes 
\[\frac{\sigma-n}{2}>\sigma-\alpha .\]
This forces $\alpha$ to be very large and so it gives rise to a somewhat uninteresting result.
\end{rem}

\begin{rem} Even if one could somehow control some $H^s_x$ norm of the solution globally in time, one may not easily iterate this local well-posedness result towards global well-posedness. This is precisely because of the memory effect, which also manifests itself in the lack of time-translation invariance. In other words, suppose we solve equation \cref{linear} for initial data $f=u(0)$ in an interval $[0,T]$ given by \cref{maintheorem}, and let that solution be $u(t)$. Then consider the IVP  \cref{linear} for initial data $f=u(T)$ this time, and its solution $v$ in some small time interval. If $\beta$ were 1, we would expect $v(t)=u(t+T)$ to hold in this interval of existence, thereby extending the lifespan of our solution. However, this fails for $\beta<1$. Instead, the right equation in the second step would be 
\begin{equation}\label{timeshifted}
\left\{ \begin{array}{ll}
i^{\beta} \prescript{}{-T}{\partial_t^{\beta}} v & = (-\Delta_x)^{\alpha /2}\, v + g(v) \qquad (t,x)\in (0,\infty )\times\RR , \\
v\mid_{t=0} & = u(T) .
\end{array}\right.
\end{equation}
where $\prescript{}{-T}{\partial_t^{\beta}} v$, is a different version of the Caputo derivative:
\[\prescript{}{-T}{\partial_t^{\beta}} v (t,x)=\frac{1}{\Gamma (1-\beta)} \int_{-T}^{t} \frac{\partial_{\tau}v(\tau,x)}{(t-\tau)^{\beta}} \, d\tau .\]
Unfortunately, it is not clear that solving \cref{timeshifted} produces an advantage over dealing with \cref{linear} directly, and therefore more research in this direction might be necessary.
\end{rem}

\begin{rem} The requirement that $p$ be an integer  in \cref{illtheorem} is not necessary. For noninteger $p$, instead of the initial data-to-solution map, one should speak about the second iterate of the following iteration scheme:
\[\left\{
    \begin{array}{ll}
    i^{\beta}\partial_t^{\beta} u_{k+1} & = (-\Delta_x)^{\frac{\alpha}{2}}u_{k+1} + \mu|u_k|^{p-1} u_k,\quad (t,x)\in [0,T]\times\RR,\\
    u_{k+1}|_{t=0} & = f\in H^s(\RR) ,
    \end{array}
\right.\]
where $u_1$ is the linear flow. Then the statement of \cref{illtheorem} becomes the following:
 when $p\geq 3$, the map from initial data in $H^s(\RR)$ to the second iterate in $C_t([0,T],H^s_x(\RR))$ is not continuous whenever $s<s_c=\frac{1}{2}-\frac{\alpha}{p-1}$. For $1<p<3$, this map will not be continuous for any $s$.
\end{rem}

\begin{rem} One might wonder if an ill-posedness result based on a small dispersion approach to show phase decoherence would be possible, in the spirit of the work in \cite{CCT}. There are several complications with this approach, such as the lack of symmetries available in our equation, the loss of derivatives, and the need for a better understanding of the behavior of the solution to the no-dispersion fractional ODE.
\end{rem}

\subsection{Notation}

We will denote by $A\lesssim B$ an estimate of the form $A\leq C B$ for some constant $B$ that might change from line to line. Similarly,  $A\lesssim_d B$ means that the implicit constant $C$ depends on $d$. 

We introduce the notation $a-$ to denote the number $a-\varepsilon$ for $0<\varepsilon \ll 1$ small enough. Similarly, we denote by $a+$ the number $a+\varepsilon$ for $0<\varepsilon \ll 1$ small enough.

For $1\leq p,q\leq \infty$ and $u: \RR \times [0,T] \longrightarrow \CC$, we define 
\[\norm{u}_{L^p_x L^q_T}=\left(\int_{\RR} \left(\int_0^T |u(t,x)|^q \, dt\right)^{\frac{p}{q}} \, dx\right)^{\frac{1}{p}} ,\]
and also 
\[\norm{u}_{ L^q_T L^p_x }=\left(\int_0^T  \left(\int_{\RR} |u(t,x)|^p \, dx\right)^{\frac{q}{p}} \, dt\right)^{\frac{1}{q}} ,\]
with the usual modifications when $p$ or $q=\infty$.
For $u: \RR \times [0,\infty) \longrightarrow \CC$, we will use the notation $L^p_x L^q_t$ and $L^q_t L^p_x$ instead, meaning $T=\infty$. We will also write $L^{p}_{T,x}$ in the case $p=q$.

We also use the standard notation for the spatial Fourier transform 
\[\widehat{f}(\xi)=\int_{\RR} f(x) e^{-ix\cdot \xi} \, dx ,\]
as well as $f^{\vee}$ for the inverse Fourier transform. The fractional Laplacian will be given by the Fourier multiplier:
\[\widehat{(-\Delta_x)^{\frac{s}{2}}f}(\xi)=\widehat{|\nabla |^{s} f}(\xi)=|\xi |^{s} \widehat{f}(\xi ).\]
 Similarly, 
 \[\widehat{\langle\nabla\rangle^s f}(\xi)=(1+|\xi|)^{s} \widehat{f}(\xi ).\]
The following notation will be used for some Sobolev norms:
\begin{align*}
\norm{f}_{\dot{H}^s(\RR)}& := \norm{|\nabla |^s f}_{L^2(\RR)},\\
\norm{f}_{H^s(\RR)}& :=\norm{\langle\nabla\rangle^s f}_{L^2(\RR)}.
\end{align*}
Finally, $C([0,T],H^s (\RR))$ denotes the space of continuous functions $u$ from a time interval $[0,T]$ to $H^s(\RR)$ equipped with the norm $\max_{t\in [0,T]}\norm{u(t,\cdot)}_{H^s_x(\RR )}$.

Additionally, we compile the following list of symbols and parameters that will be used along the paper:
\begin{itemize}
\item $\alpha$ - the order of the fractional Laplacian $(-\Delta_x)^{\frac{\alpha}{2}}$.
\item $\beta$  - the order of the Caputo derivative in time $\partial_t^{\beta}$.
\item $\sigma$ - the ratio $\frac{\alpha}{\beta}$.
\item $p$ - the degree of the power-type nonlinearity.
\item $\gamma=\frac{\sigma-1}{2}$ - the gain in the linear part thanks to the smoothing effect.
\item $\tilde{\gamma}=\alpha- \frac{\sigma+1}{2}$ - the gain in the nonlinear part thanks to the smoothing effect.
\end{itemize}

\subsection{Outline}
The paper is organized as follows. In \cref{sec: linear estimates}, we prove linear estimates with respect to the norms given in \cref{class1,class2,class3}. In order to do this, we use a representation of the solution given by a Fourier multiplier for which we only have asymptotic formulas, and thus it requires a somewhat different treatment on small and large frequencies, as well as studying a remainder. In \cref{sec: easy proof}, we employ these linear estimates to prove \cref{maincor}, which is similar to \cref{maintheorem} for a special but illustrative choice of parameters, which simplifies the exposition. In \cref{sec: ill-posedness}, we prove \cref{illtheorem}.
In \cref{sec: appA}, we present additional information about Caputo fractional derivatives and explain our representation of the solutions. In \cref{sec: appB}, we provide the proof of \cref{maintheorem} in full generality.

\section{Linear estimates}
\label{sec: linear estimates}

For $f:\RR \longrightarrow \CC$, we introduce the following operators:
\begin{align}
S_t f(x) & = \int_{\RR} e^{-i t |\xi |^{\sigma}}\, \chi_{\{t|\xi |^{\sigma} \leq M\}} \widehat{f}(\xi) \, e^{ix\xi} d\xi , \nonumber\\
T_t f(x) & = \int_{\RR} t^{-\beta} |\xi |^{-\alpha} \chi_{\{t|\xi |^{\sigma} > M\}} \widehat{f}(\xi) \, e^{ix\xi} d\xi ,\label{Loperators}\\
U_t f(x) & = \int_{\RR} E_{\beta} (i^{-\beta}t^{\beta} |\xi |^{\alpha})\, \chi_{\{t|\xi |^{\sigma} \leq M\}} \widehat{f}(\xi) \, e^{ix\xi} d\xi ,\nonumber
\end{align}
where $\chi_{\{t|\xi |^{\sigma} \leq M\}}=\chi (t,\xi )$ denotes a smooth function supported on the set $\{ (t,\xi)\in (0,\infty)\times \RR \mid t|\xi |^{\sigma} \leq 2M\}$ for some large $M>0$, satisfying $\chi (t,\xi )=1$ if  $t|\xi |^{\sigma} \leq M$. Similarly, we will denote $\chi_{\{t|\xi |^{\sigma} > M\}}:=1-\chi_{\{t|\xi |^{\sigma} \leq M\}}$.

The operators $S_t$ and $U_t$ will capture the behavior of the solution \cref{Lsolution} when $t|\xi |^{\sigma}$ is small, whereas $e^{it|\nabla |^{\sigma}}$ and $T_t$ will do so for large $t|\xi |^{\sigma}$, based on the first and second terms of the asymptotic formula \cref{Lasymp}.

Similarly, the following multiplier operators will also play a role when dealing with the nonlinearity:
\begin{align}
\tilde{S}_t f(x) & = \int_{\RR} |\xi |^{\sigma -\alpha} \, e^{-i t |\xi |^{\sigma}}\, \chi_{\{t|\xi |^{\sigma} \leq M\}} \widehat{f}(\xi) \, e^{ix\xi} d\xi ,\nonumber\\
\tilde{T}_t f(x) & = \int_{\RR} t^{-1-\beta} |\xi |^{-2\alpha} \chi_{\{t|\xi |^{\sigma} > M\}} \widehat{f}(\xi) \, e^{ix\xi} d\xi ,\label{NLoperators}\\
\tilde{U}_t f(x) & = \int_{\RR} t^{\beta-1} \, E_{\beta , \beta} (i^{-\beta}t^{\beta} |\xi |^{\alpha})\, \chi_{\{t|\xi |^{\sigma} \leq M\}} \widehat{f}(\xi) \, e^{ix\xi} d\xi  .\nonumber
\end{align}

As before, the operators $\tilde{S}_t$ and $\tilde{U}_t$ will capture the behavior of the second term in \cref{inteq} when $t|\xi |^{\sigma}$ is small. For large values of $t|\xi |^{\sigma}$, the operators $|\nabla |^{\sigma-\alpha} e^{it|\nabla |^{\sigma}}$ and $\tilde{T}_t$ will be used, based on the first and second terms in \cref{NLasymp}.

\subsection{$L^{\infty}_x L^2_T$ estimates - smoothing effect}

The following two propositions are a generalization of Theorem 3.5 in \cite{KPV}.

\begin{proposition}\label{smoothingprop} Let $\gamma=\frac{\sigma -1}{2}\geq 0$. Then 
\begin{equation}\label{smoothingeffect}
\norm{|\nabla |^{\gamma} e^{-it |\nabla |^{\sigma} } f}_{L^{\infty}_x L^2_t}\lesssim \norm{f}_{L^2_x} \ .
\end{equation}
\end{proposition}
\begin{proof}
We define
\[ F(t,x):= \int_{\RR} |\xi |^{\gamma} e^{-it |\xi|^{\sigma}} \widehat{f}(\xi) e^{ix\xi} \, d\xi .\]
Then we can rewrite this as:
\begin{align*}
F(t,x) & =\int_0^{\infty} \xi^{\gamma} e^{-it \xi^{\sigma}} \left( \widehat{f}(\xi ) e^{ix\xi} + \widehat{f}(-\xi )e^{-ix\xi}\right) d\xi = \int_0^{\infty} \xi^{\gamma} e^{-it \xi^{\sigma}}  \widehat{f}^{\star}(x,\xi) d\xi \\
& = \frac{1}{\sigma}\int_0^{\infty} \mu^{\frac{\gamma+1}{\sigma}-1}\widehat{f}^{\star}(x,\mu^{\frac{1}{\sigma}}) e^{-it\mu} \, d\mu .
\end{align*}
By the Plancherel identity,
\begin{align*}
\norm{F(x)}_{L^2_t}^2 & = \int_0^{\infty} \Big | \frac{1}{\sigma} \mu^{\frac{\gamma+1}{\sigma}-1}\widehat{f}^{\star}(x,\mu^{\frac{1}{\sigma}}) \Big |^2 \, d\mu \\
& = \int_{0}^{\infty} \frac{1}{\sigma} \xi^{2\gamma + 1 -\sigma } |\widehat{f}^{\star}(x,\xi)|^2 d\xi= \int_{0}^{\infty} \frac{1}{\sigma}|\widehat{f}^{\star}(x,\xi)|^2 d\xi .
\end{align*}
Finally, since $\sup_x |\widehat{f}^{\star}(x,\xi)|^2\leq 2 | \widehat{f}(\xi)|^2 + 2 | \widehat{f}(-\xi)|^2$ we obtain
\[\norm{F}_{L^{\infty}_x L^2_t}\lesssim \norm{f}_{L^2_x} .\]
\end{proof}

\begin{rem} It is easy to further generalize this proposition to any dimension $n$. In such a case, if $\gamma=\frac{\sigma -n}{2}$ then \cref{smoothingeffect} holds. As we will see later, in order to gain derivatives and close the contraction-mapping principle with the techniques employed in this paper, we need $\alpha - \sigma + \gamma>0$ to be true. This will not happen for any $n\geq 2$.
\end{rem}

From the linear estimate in \cref{smoothingprop} we obtain the following:

\begin{proposition}\label{smoothingprop2} For $\gamma=\frac{\sigma -1}{2}\geq 0$ and a fixed time $T>0$, we have
\[\norm{|\nabla |^{\gamma} \int_0^t e^{-i(t-t') |\nabla |^{\sigma} } G(t',x)  \, dt' }_{L^{\infty}_x L^2_T}\lesssim \norm{G}_{L^1_T L^2_x} . \]
\end{proposition}
\begin{proof}
The dual of the estimate \cref{smoothingeffect} is
\begin{equation}\label{dual}
\norm{|\nabla |^{\gamma} \int_{-\infty}^{\infty} e^{it' |\nabla |^{\sigma} } g(t',x)  \, dt' }_{L^2_x}\lesssim \norm{g}_{L^1_x L^2_t} .
\end{equation}

From this estimate, consider the function $\tilde{g}(t',x):=  \chi_{[t,T]}(t')\, g(t',x)$, which is clearly dominated by $g$ and therefore is also in $L^1_x L^2_t$. We substitute this in \cref{dual} and obtain
\[\norm{|\nabla |^{\gamma} \int_t^T e^{it' |\nabla |^{\sigma} } g(t',x)  \, dt' }_{L^2_x}\lesssim \norm{g}_{L^1_x L^2([t,T])} .\]

Since the operator $e^{-it|\nabla |^{\sigma}}$ is unitary,
\[\norm{e^{-it|\nabla |^{\sigma}}\, |\nabla |^{\gamma} \int_t^T e^{it' |\nabla |^{\sigma} } g(t',x)  \, dt' }_{L^2_x}\lesssim \norm{g}_{L^1_x L^2([t,T])} .\]

We now take the $L^{\infty}_T$ norm:
\begin{equation}\label{L1xL2T}
\norm{|\nabla |^{\gamma} \int_t^T e^{-i(t-t') |\nabla |^{\sigma} } g(t',x)  \, dt' }_{L^{\infty}_T L^2_x}\lesssim \norm{g}_{L^1_x L^2_T} .
\end{equation}

Finally, consider
\begin{align*}
\left\lVert |\nabla |^{\gamma} \int_0^t \right. & \left. e^{-i(t-t') |\nabla |^{\sigma} } G(t',x)  \, dt' \right\rVert_{L^{\infty}_x L^2_T} \\
& = \sup_{\norm{g}_{L^1_x L^2_T}=1} \, \Big | \int_0^T \int_{\RR} \overline{g}(t,x) \left( |\nabla |^{\gamma} \int_0^t e^{-i(t-t') |\nabla |^{\sigma} } G(t',x)  \, dt' \right) \, dx \, dt \Big | \\
&  = \sup_{\norm{g}_{L^1_x L^2_T}=1} \, \Big | \int_0^T \int_{\RR} G(t',x) \overline{\left( |\nabla |^{\gamma} \int_{t'}^T e^{-i(t'-t) |\nabla |^{\sigma} } g(t,x)  \, dt \right)} \, dx \, dt' \Big | .
\end{align*}
Now we finish by using the H\"older inequality together with \cref{L1xL2T}.
\end{proof}

\begin{proposition} For any $0\leq \gamma' <\gamma=\frac{\sigma-1}{2}$, the operator $T_t$ defined in \cref{Loperators} satisfies
\[\norm{|\nabla |^{\gamma'} \, T_t f}_{L^{\infty}_x L^2_T}\lesssim T^{\frac{\gamma - \gamma'}{\sigma}} \norm{f}_{L^2_x} . \]
\end{proposition}
\begin{proof} By the Minkowski inequality,
\begin{align*}
\norm{|\nabla |^{\gamma'} \, T_t f}_{L^{\infty}_x L^2_T} & \leq \int_{\RR} \norm{t^{-\beta} |\xi |^{\gamma'-\alpha} \chi_{\{t|\xi |^{\sigma} > M\}}\, \widehat{f}(\xi)}_{L^2_T} \, d\xi \\
& \hspace{-1cm} = \int_{\RR} |\xi |^{\gamma'-\alpha} |\widehat{f}(\xi)| \chi_{\{|\xi |^{\sigma} > \frac{M}{T}\}} \left(\int_{\frac{M}{|\xi |^{\sigma}}}^T t^{-2\beta} \, dt\right)^{\frac{1}{2}} \, d\xi \\
& \hspace{-1cm} =\int_{\RR} |\xi |^{\gamma'-\alpha} |\widehat{f}(\xi)| \chi_{\{|\xi |^{\sigma} > \frac{M}{T}\}} \left[ \frac{T^{1-2\beta}}{1-2\beta} - \frac{1}{1-2\beta} \left(\frac{M}{|\xi |^{\sigma}}\right)^{1-2\beta} \right]^{\frac{1}{2}} \, d\xi .
\end{align*}
We shall see later that the assumption $\beta >\frac{1}{2}$ will happen naturally, so now observe that because of the characteristic function we have the following bound
\begin{align*}
\norm{|\nabla |^{\gamma'} \, T_t f}_{L^{\infty}_x L^2_T} & \lesssim_{M, \beta} \int_{\RR} |\xi |^{\gamma'-\alpha} |\widehat{f}(\xi)| \chi_{\{|\xi |^{\sigma} > \frac{M}{T}\}} |\xi |^{\alpha-\frac{\sigma}{2} } \, d\xi \\
& = \int_{\RR} |\xi |^{\gamma'-\frac{\sigma}{2}} |\widehat{f}(\xi)| \chi_{\{|\xi |^{\sigma} > \frac{M}{T}\}} \, d\xi \\
& \leq \norm{f}_{L^2_x} \, \left(\int_{\RR} |\xi |^{2\gamma'-\sigma}\chi_{\{|\xi |^{\sigma} > \frac{M}{T}\}} \, d\xi\right)^{\frac{1}{2}} .
\end{align*}
The condition $2\gamma'-\sigma<-1$ is necessary to ensure integrability in the previous inequality.
\end{proof}

Finally we look at what happens for small frequencies. 

\begin{proposition} For any $0\leq \gamma' <\gamma=\frac{\sigma-1}{2}$, the operator $U_t$ defined in \cref{Loperators} satisfies
\[\norm{|\nabla |^{\gamma'} \, U_t f}_{L^{\infty}_x L^2_T}\lesssim T^{\frac{\gamma - \gamma'}{\sigma}} \norm{f}_{L^2_x} . \]
\end{proposition}
\begin{proof} Let us remember the definition of this operator
\[U_t f(x) = \int_{\RR} E_{\beta} (i^{-\beta}t^{\beta} |\xi |^{\alpha})\, \chi_{\{t|\xi |^{\sigma} \leq M\}} \widehat{f}(\xi) \, e^{ix\xi} d\xi ,\]
where $\sigma=\frac{\alpha}{\beta}$. Note that the series defining the function $E_{\beta}$ is absolutely convergent in the support of $\chi_{\{t|\xi |^{\sigma} \leq M\}}$, see \cref{ML}. Using this fact together with the Minkowski inequality we find
\begin{align*}
\norm{|\nabla |^{\gamma'} \, U_t f}_{L^{\infty}_x L^2_T} \lesssim & \int_{\RR} |\widehat{f}(\xi)| |\xi |^{\gamma'} \left( \int_0^{\min (T, \frac{M}{|\xi |^{\sigma}})} dt\right)^{\frac{1}{2}} \, d\xi \\
 & \hspace{-1.5cm} \lesssim \int_{\RR} |\widehat{f}(\xi)| |\xi |^{\gamma'-\frac{\sigma}{2}} \chi_{\{T |\xi |^{\sigma} > M\}} \, d\xi 
 + T^{\frac{1}{2}} \int_{\RR} |\widehat{f}(\xi)| |\xi |^{\gamma'} \chi_{\{T |\xi |^{\sigma} \leq M\}} \, d\xi \\
  & \hspace{-1.5cm} \lesssim \norm{f}_{L^2_x} \, \left( \int_{\RR} |\xi |^{2\gamma'-\sigma}\, \chi_{\{T |\xi |^{\sigma} > M\}} \, d\xi \right)^{\frac{1}{2}} \\
& \hspace{-1.4cm} +\norm{f}_{L^2_x} \, T^{\frac{1}{2}} \, \left( \int_{\RR} |\xi |^{2\gamma'} \chi_{\{T |\xi |^{\sigma} \leq M\}} \, d\xi \right)^{\frac{1}{2}} ,
\end{align*}
Again, one sees that the condition $2\gamma' - \sigma < -1$ is required to ensure integrability, and the exponent in $T$ coincides for both integrals.
\end{proof}

The same argument proves the following result:

\begin{proposition} For any $0\leq \gamma' <\gamma=\frac{\sigma-1}{2}$, the operator $S_t$ defined in \cref{Loperators} satisfies
\[\norm{|\nabla |^{\gamma'} \, S_t f}_{L^{\infty}_x L^2_T}\lesssim T^{\frac{\gamma - \gamma'}{\sigma}} \norm{f}_{L^2_x} . \]
\end{proposition}

We now rewrite \cref{smoothingprop,smoothingprop2} in a clearer way for the purpose of estimating what will be the nonlinear part of the equation.

\begin{proposition} Let $\tilde{\gamma}=\alpha - \frac{\sigma +1}{2}$. Then 
\begin{align*}
\norm{|\nabla |^{\tilde{\gamma}} |\nabla |^{\sigma - \alpha} e^{-it |\nabla |^{\sigma} } f}_{L^{\infty}_x L^2_t}& \lesssim \norm{f}_{L^2_x} , \\
\norm{|\nabla |^{\tilde{\gamma}} \int_0^t |\nabla |^{\sigma - \alpha} e^{-i(t-t') |\nabla |^{\sigma} } G(t',x)  \, dt' }_{L^{\infty}_x L^2_T} & \lesssim \norm{G}_{L^1_T L^2_x} . 
\end{align*}
\end{proposition}

Note that the only way to gain derivatives is if $\tilde{\gamma}>0$, which in particular implies that $\alpha >1$ and $\beta>\frac{1}{2}$.

Similarly, we need the analog of this result for the other multiplier operators taking part in the nonlinear piece.

\begin{proposition}\label{need} For any $0\leq \tilde{\gamma}' <\tilde{\gamma}=\alpha - \frac{\sigma +1}{2}$ the operator $\tilde{T}_t$ defined in \cref{NLoperators} satisfies
\[\norm{|\nabla |^{\tilde{\gamma}'}\int_0^t  \, \tilde{T}_{t-t'} G(t',x) \, dt'}_{L^{\infty}_x L^2_T}\lesssim T^{\frac{\tilde{\gamma} - \tilde{\gamma}'}{\sigma}} \norm{G}_{L^1_T L^2_x} . \]
\end{proposition}
\begin{proof} By the Minkowski inequality,
\begin{align}
\left\lVert |\nabla |^{\tilde{\gamma}'}\right. & \left.\int_0^t \tilde{T}_{t-t'} G(t',x) \, dt'\right\rVert_{L^{\infty}_x L^2_T} \nonumber \\
 &  \leq \int_0^T \int_{\RR} |\xi |^{\tilde{\gamma}'-2\alpha}\, |\widehat{G}(t',\xi)|\, \norm{ \frac{\chi_{\{t'\leq t\}}}{(t-t')^{1+\beta}} \, \chi_{\{(t-t')|\xi |^{\sigma} >M\}} }_{L^2_t ([0,T])} \, d\xi dt' \nonumber \\
 &  \lesssim \int_0^T \int_{\RR} |\xi |^{\tilde{\gamma}'-2\alpha}\, |\widehat{G}(t',\xi)|\, \chi_{\{(T-t')|\xi |^{\sigma} >M\}} \left( \frac{|\xi |^{\sigma + 2\alpha}}{M^{1+2\beta}} - \frac{1}{(T-t')^{1+2\beta}}\right)^{\frac{1}{2}} \, d\xi dt' \label{nonhom}\\
&  \lesssim \int_0^T \int_{\RR} |\xi |^{\tilde{\gamma}'-\alpha+\frac{\sigma}{2}}\, |\widehat{G}(t',\xi)|\, \chi_{\{(T-t')|\xi |^{\sigma} >M\}} d\xi dt' \nonumber \\
&  \lesssim \int_0^T \norm{G(t)}_{L^2_x} \, \left( \int_{\RR} |\xi |^{2\tilde{\gamma}'-2\alpha+\sigma}\,\chi_{\{(T-t')|\xi |^{\sigma} >M\}} d\xi \right)^{\frac{1}{2}} \, dt' \nonumber \\
&  \lesssim \int_0^T \norm{G(t)}_{L^2_x} \, (T-t')^{ \frac{\tilde{\gamma} - \tilde{\gamma}'}{\sigma}} dt' \leq T^{\frac{\tilde{\gamma} - \tilde{\gamma}'}{\sigma}} \norm{G}_{L^1_T L^2_x} .\nonumber 
\end{align}
\end{proof}

The proofs of the following results are analogous to the ones before and so we will omit them.

\begin{proposition} For any $0\leq \tilde{\gamma}' <\tilde{\gamma}=\alpha - \frac{\sigma +1}{2}$ the operator $\tilde{U}_t$ defined in \cref{NLoperators} satisfies
\[\norm{|\nabla |^{\tilde{\gamma}'}\int_0^t  \, \tilde{U}_{t-t'} G(t',x) \, dt'}_{L^{\infty}_x L^2_T}\lesssim T^{\frac{\tilde{\gamma} - \tilde{\gamma}'}{\sigma}} \norm{G}_{L^1_T L^2_x} . \]
\end{proposition}

\begin{proposition} For any $0\leq \tilde{\gamma}' <\tilde{\gamma}=\alpha - \frac{\sigma +1}{2}$ the operator $\tilde{S}_t$ defined in \cref{NLoperators} satisfies
\[\norm{|\nabla |^{\tilde{\gamma}'}\int_0^t  \, \tilde{S}_{t-t'} G(t',x) \, dt'}_{L^{\infty}_x L^2_T}\lesssim T^{\frac{\tilde{\gamma} - \tilde{\gamma}'}{\sigma}} \norm{G}_{L^1_T L^2_x} . \]
\end{proposition}

Now it is time to put all these results together:

\begin{theorem}[Norm homogeneous version]\label{smooth} Let $0\leq \gamma' <\gamma=\frac{\sigma-1}{2}$, and $0\leq \tilde{\gamma}' <\tilde{\gamma}=\alpha - \frac{\sigma +1}{2}$. Then 
\begin{align*}
\norm{|\nabla|^{\gamma} e^{-it|\nabla |^{\sigma}} f}_{L^{\infty}_x L^2_T}& \lesssim \norm{f}_{L^2_x} ,\\
\norm{|\nabla|^{\gamma'} S_t f}_{L^{\infty}_x L^2_T} & \lesssim T^{\frac{\gamma - \gamma'}{\sigma}} \norm{f}_{L^2_x} ,\\
\norm{|\nabla|^{\gamma'} T_t f}_{L^{\infty}_x L^2_T} & \lesssim T^{\frac{\gamma - \gamma'}{\sigma}} \norm{f}_{L^2_x} ,\\
\norm{|\nabla|^{\gamma'} U_t f}_{L^{\infty}_x L^2_T} & \lesssim T^{\frac{\gamma - \gamma'}{\sigma}} \norm{f}_{L^2_x} .
\end{align*}
Additionally,
\begin{align*}
\norm{|\nabla|^{\tilde{\gamma}} \int_0^t |\nabla|^{\sigma-\alpha} e^{-i(t-t')|\nabla |^{\sigma}} G(t',x) dt'}_{L^{\infty}_x L^2_T}& \lesssim \norm{G}_{L^1_T L^2_x} ,\\
\norm{|\nabla|^{\tilde{\gamma}'} \int_0^t \tilde{S}_{t-t'} G(t',x) dt'}_{L^{\infty}_x L^2_T} & \lesssim T^{\frac{\tilde{\gamma} - \tilde{\gamma}'}{\sigma}} \norm{G}_{L^1_T L^2_x} ,\\
\norm{|\nabla|^{\tilde{\gamma}'} \int_0^t \tilde{T}_{t-t'} G(t',x) dt'}_{L^{\infty}_x L^2_T} & \lesssim T^{\frac{\tilde{\gamma} - \tilde{\gamma}'}{\sigma}} \norm{G}_{L^1_T L^2_x} ,\\
\norm{|\nabla|^{\tilde{\gamma}'} \int_0^t \tilde{U}_{t-t'} G(t',x) dt'}_{L^{\infty}_x L^2_T} & \lesssim T^{\frac{\tilde{\gamma} - \tilde{\gamma}'}{\sigma}} \norm{G}_{L^1_T L^2_x} .
\end{align*}
\end{theorem}

\begin{theorem}[Norm nonhomogeneous version]\label{thsmooth} Let $0\leq \gamma' <\gamma=\frac{\sigma-1}{2}$, and $0\leq \tilde{\gamma}' <\tilde{\gamma}=\alpha - \frac{\sigma +1}{2}$. Then 
\begin{align*}
\norm{\langle\nabla\rangle^{\gamma} e^{-it|\nabla |^{\sigma}} f}_{L^{\infty}_x L^2_T}& \lesssim (1+T^{\frac{1}{2}}) \norm{f}_{L^2_x} ,\\
\norm{\langle\nabla\rangle^{\gamma'} S_t f}_{L^{\infty}_x L^2_T} & \lesssim (T^{\frac{\gamma - \gamma'}{\sigma}} +T^{\frac{1}{2}}) \norm{f}_{L^2_x} ,\\
\norm{\langle\nabla\rangle^{\gamma'} T_t f}_{L^{\infty}_x L^2_T} & \lesssim (T^{\frac{\gamma - \gamma'}{\sigma}} +T^{\frac{1}{2}}) \norm{f}_{L^2_x} ,\\
\norm{\langle\nabla\rangle^{\gamma'} U_t f}_{L^{\infty}_x L^2_T} & \lesssim (T^{\frac{\gamma - \gamma'}{\sigma}} +T^{\frac{1}{2}}) \norm{f}_{L^2_x} .
\end{align*}
Moreover,
\begin{align*}
\norm{\langle\nabla\rangle^{\tilde{\gamma}} \int_0^t |\nabla|^{\sigma-\alpha} e^{-i(t-t')|\nabla |^{\sigma}} G(t',x) dt'}_{L^{\infty}_x L^2_T}& \lesssim (1 +T^{\frac{1}{2}}) \norm{G}_{L^1_T L^2_x} ,\\
\norm{\langle\nabla\rangle^{\tilde{\gamma}'} \int_0^t \tilde{S}_{t-t'} G(t',x) dt'}_{L^{\infty}_x L^2_T} & \lesssim (T^{\frac{\tilde{\gamma} - \tilde{\gamma}'}{\sigma}} +T^{\frac{1}{2}}) \norm{G}_{L^1_T L^2_x} ,\\
\norm{\langle\nabla\rangle^{\tilde{\gamma}'} \int_0^t \tilde{T}_{t-t'} G(t',x) dt'}_{L^{\infty}_x L^2_T} & \lesssim  T^{\frac{\tilde{\gamma} - \tilde{\gamma}'}{\sigma}} \, \norm{G}_{L^1_T L^2_x} ,\\
\norm{\langle\nabla\rangle^{\tilde{\gamma}'} \int_0^t \tilde{U}_{t-t'} G(t',x) dt'}_{L^{\infty}_x L^2_T} & \lesssim (T^{\frac{\tilde{\gamma} - \tilde{\gamma}'}{\sigma}} +T^{\frac{1}{2}}) \norm{G}_{L^1_T L^2_x}  .
\end{align*}
\end{theorem}
\begin{proof}
We write $f= P_{\leq 1} f + P_{>1} f$ where $(P_{\leq 1} f)^{\wedge}=\phi (\xi) \widehat{f}(\xi)$ for $\phi\in C_c^{\infty}(\RR)$ and supp$(\phi)\subset B(0,1)$, and $(P_{>1} f)^{\wedge} =(1-\phi (\xi)) \widehat{f}(\xi)$.

Take for example
\[\norm{\langle\nabla\rangle^{\gamma} e^{-it|\nabla |^{\sigma}} f}_{L^{\infty}_x L^2_T} \leq \norm{\langle\nabla\rangle^{\gamma} e^{-it|\nabla |^{\sigma}} P_{\leq 1} f}_{L^{\infty}_x L^2_T} + \norm{\langle\nabla\rangle^{\gamma} e^{-it|\nabla |^{\sigma}} P_{>1} f}_{L^{\infty}_x L^2_T} \ .\]
Let $\tilde{f}=\langle\nabla\rangle^{\gamma}P_{> 1} f$. Then by \cref{smoothingprop} and the fact that these operators commute,
\begin{align*}
\norm{\langle\nabla\rangle^{\gamma} e^{-it|\nabla |^{\sigma}} P_{> 1} f}_{L^{\infty}_x L^2_T} & =\norm{ e^{-it|\nabla |^{\sigma}} \tilde{f}}_{L^{\infty}_x L^2_T}\lesssim \norm{|\nabla |^{-\gamma} \tilde{f}}_{L^2_x} \\
& =\norm{|\nabla |^{-\gamma} \langle\nabla\rangle^{\gamma}P_{> 1} f}_{L^2_x}\lesssim \norm{f}_{L^2_x} ,
\end{align*}
which follows from the fact that $|\nabla |^{-\gamma} \langle\nabla\rangle^{\gamma}P_{> 1}$ is a bounded operator from $L^2_x$ to $L^2_x$.

On the other hand, by the Minkowski inequality
\begin{align}\label{eq:minkowski}
\norm{\langle\nabla\rangle^{\gamma} e^{-it|\nabla |^{\sigma}} P_{\leq 1} f}_{L^{\infty}_x L^2_T} & =\norm{\int_{\RR} e^{-it|\xi |^{\sigma}} e^{ix\xi} \phi (\xi) \langle \xi\rangle^{\gamma} \widehat{f}(\xi ) \, d\xi}_{L^{\infty}_x L^2_T} \\
&\lesssim T^{\frac{1}{2}} \, \int_{\RR} | \phi (\xi ) \widehat{f}(\xi )| \, d\xi \lesssim T^{\frac{1}{2}} \,\norm{\widehat{f}\, }_{L^2_x} \lesssim T^{\frac{1}{2}}\, \norm{f}_{L^2_x} .\nonumber
\end{align}

The idea for the others will be the same, since all $S_t$, $T_t$ and $U_t$ composed with $P_{\leq 1}$ are bounded.

Now let's look at the estimates that we will use on the nonlinear part. Take for instance 
\[\norm{\langle\nabla\rangle^{\tilde{\gamma}'} \int_0^t \tilde{U}_{t-t'} G(t',x) dt'}_{L^{\infty}_x L^2_T} \lesssim (T^{\frac{\tilde{\gamma} - \tilde{\gamma}'}{\sigma}} +T^{\frac{1}{2}}) \norm{G}_{L^1_T L^2_x},\]
and let's prove it is true for $P_{\leq 1} G$ and $P_{>1} G$. \Cref{need} yields:
\[ \norm{\int_0^t \tilde{U}_{t-t'} (P_{>1} G)(t',x) dt'}_{L^{\infty}_x L^2_T} \lesssim  T^{\frac{\tilde{\gamma} - \tilde{\gamma}'}{\sigma}} \norm{G}_{L^1_T L^2_x}, \]
as explained before.

An argument such as \cref{eq:minkowski} based on repeated use of the Minkowski inequality yields  
\[ \norm{\int_0^t \tilde{U}_{t-t'} (P_{\leq1} G)(t',x) dt'}_{L^{\infty}_x L^2_T} \lesssim  T^{\frac{1}{2}}\, \norm{G}_{L^1_T L^2_x}. \]

Note that this step is not necessary for $\tilde{T}_{t-t'}$ because it is supported on the set $\{(t-t')|\xi|^{\sigma}>M\}$ which does not intersect $\{ |\xi|\leq 1\}$ for $M>T$.
\end{proof}

\subsection{$L^{\infty}_T L^2_x$ estimates}

\begin{proposition}\label{L21} Using the definitions for $S_t$, $T_t$ and $U_t$ given in \cref{Loperators}, we have
\begin{align*}
\norm{e^{-it|\nabla |^{\sigma}} f}_{L^{\infty}_T L^2_x} & \lesssim \norm{f}_{L^2_x} ,\\
\norm{S_t f}_{L^{\infty}_T L^2_x} &  \lesssim \norm{f}_{L^2_x} ,\\
\norm{T_t f}_{L^{\infty}_T L^2_x} & \lesssim \norm{f}_{L^2_x} ,\\
\norm{U_t f}_{L^{\infty}_T L^2_x} & \lesssim \norm{f}_{L^2_x} .
\end{align*}
\end{proposition}
\begin{proof}
We use the Plancherel theorem, and then the fact that every one of the multipliers involved are bounded.
\end{proof}

The Minkowski inequality immediately implies:

\begin{proposition}\label{L22} The following estimate holds:
\[\norm{\int_0^t |\nabla |^{\sigma -\alpha} e^{-i(t-t')|\nabla |^{\sigma}} G(t',x)\, dt'}_{L^{\infty}_T L^2_x} \lesssim \norm{|\nabla |^{\sigma - \alpha} G}_{L^1_T L^2_x} .\]
\end{proposition}

\begin{proposition}\label{L23} For $\tilde{S}_t$, $\tilde{T}_t$ and $\tilde{U}_t$ as defined in \cref{NLoperators} we have
\begin{align*}
\norm{\int_0^t \tilde{S}_{t-t'} G(t',x)\, dt'}_{L^{\infty}_T L^2_x} & \lesssim T^{\beta - \frac{1}{2}}\norm{G}_{L^2_{T,x}} ,\\
\norm{\int_0^t \tilde{T}_{t-t'} G(t',x)\, dt'}_{L^{\infty}_T L^2_x} & \lesssim T^{\beta - \frac{1}{2}}\norm{G}_{L^2_{T,x}} ,\\
\norm{\int_0^t \tilde{U}_{t-t'} G(t',x)\, dt'}_{L^{\infty}_T L^2_x} & \lesssim T^{\beta - \frac{1}{2}}\norm{G}_{L^2_{T,x}} .
\end{align*}
\end{proposition}
\begin{proof}
For $\tilde{T_t}$, we use the Minkowski inequality, the Plancherel theorem, then the cut-off and finally the Cauchy-Schwarz inequality:
\begin{align*}
\Big\lVert\int_0^t \tilde{T}_{t-t'} & G(t',x)\, dt'\Big \rVert_{L^{\infty}_T L^2_x} \\
& \lesssim \sup_{0\leq t\leq T} \int_0^t \norm{ (t-t')^{-1-\beta} |\xi |^{-2\alpha} |\widehat{G}(t',\xi)| \chi_{\{(t-t')|\xi |^{\sigma}>M\}}}_{L^2_{\xi}} \, dt' \\
& \lesssim \sup_{0\leq t\leq T} \int_0^t (t-t')^{-1+\beta} \norm{G(t')}_{L^2_x} \, dt' \\
& \lesssim \norm{G}_{L^2_T L^2_x} \, \sup_{0\leq t\leq T} \left(\int_0^t (t-t')^{-2+2\beta} \, dt'\right)^{\frac{1}{2}} .
\end{align*}
The proofs for the corresponding inequalities involving $\tilde{S}_t$ and $\tilde{U}_t$ are analogous.
\end{proof}

\subsection{$L^{p}_x L^{\infty}_T$ estimates - maximal function}

Now we are interested in studying what happens with these operators in spaces like $L^p_x L^{\infty}_T$ for $p\geq 2$. 
For the oscillatory part, there are global estimates that we may use. The following result can be found in \cite{KPV} as Lemma 3.29, and also as Theorem 1 in \cite{Peter}.

\begin{proposition}\label{maxfunction} Suppose that $\sigma >1$ and $f\in\mathcal{S}(\RR)$, the Schwartz space. The inequality
\[ \norm{e^{-it|\nabla |^{\sigma}}f}_{L^p_x L^{\infty}_t} \lesssim \norm{|\nabla |^s f}_{L^2_x}\]
holds if and only if $p=\frac{2}{1-2s}$ and $\frac{1}{4}\leq s <\frac{1}{2}$.
\end{proposition}

The other operators can be treated as follows.

\begin{proposition} For $p>2$ and $T_t$ as defined in \cref{Loperators}, we have
\[ \norm{T_t f}_{L^p_x L^{\infty}_t } \lesssim \, \norm{|\nabla |^s \, f}_{L^2_x} ,\]
where $s=\frac{1}{2}-\frac{1}{p}$.
\end{proposition}
\begin{proof}
We define 
\[\mathcal{T}_t(x):=\int_{\RR} t^{-\beta} |\xi |^{-\alpha} \chi_{\{t|\xi |^{\sigma} > M\}} \, e^{ix\xi} d\xi,\]
 which is well-defined because we are working under the assumption that $\alpha >1$. Recall that $\sigma=\frac{\alpha}{\beta}$ and consider 
\[|\nabla |^{-s} \mathcal{T}_t(x)=\int_{\RR} t^{-\beta} |\xi |^{-s-\alpha} \chi_{\{t|\xi |^{\sigma} > M\}} \, e^{ix\xi} d\xi, \]
Our goal is to prove that
\begin{equation}\label{s-decay}
\Big | |\nabla |^{-s} \mathcal{T}_t(x)\Big | \lesssim_s |x|^{s-1} \ \mbox{for all}\ x\in\RR,
\end{equation}
and for any $s\in [0,1]$ independently of time. By using 
\[|\nabla |^{-s} \mathcal{T}_t(x)=t^{\frac{s-1}{\sigma}} (|\nabla |^{-s} \mathcal{T}_1 )(t^{-\frac{1}{\sigma}} x),\] which follows by rescaling, we can rewrite inequality \cref{s-decay} as
\begin{align}
\Big |\, t^{\frac{s-1}{\sigma}} (|\nabla |^{-s} \mathcal{T}_1 )(t^{-\frac{1}{\sigma}} x)\Big | & \lesssim_s |x|^{s-1} ,\nonumber\\
 \Big |(|\nabla |^{-s} \mathcal{T}_1 )(t^{-\frac{1}{\sigma}} x)\Big | & \lesssim_s t^{\frac{1-s}{\sigma}} |x|^{s-1}= |t^{-\frac{1}{\sigma}} x |^{s-1} ,\nonumber\\
\Big | |\nabla |^{-s} \mathcal{T}_1(x)\Big | & \lesssim_s |x|^{s-1} .\label{HLS}
\end{align}
And therefore, it is enough to prove \cref{s-decay} in the case $t=1$, namely 
\[|\nabla |^{-s} \mathcal{T}_1(x)=\int_{\RR} |\xi |^{-s-\alpha} \chi_{\{|\xi |^{\sigma} > M\}} \, e^{ix\xi} d\xi . \]
By absolute integrability we have $\Big | |\nabla |^{-s} \mathcal{T}_1(x)\Big | \lesssim_s 1$, which we will use when $|x|$ is small. For $|x|>2$, we integrate by parts:
 \begin{align*}
 |\nabla |^{-s} \mathcal{T}_1(x) & =\int_{\RR} |\xi |^{-s-\alpha} \chi_{\{|\xi |^{\sigma} > M\}} \, \frac{1}{ix}\frac{d}{ d\xi} e^{ix\xi} d\xi \\
 & = - \int_{\RR} \frac{1}{ix}\frac{d}{ d\xi}\left( |\xi |^{-s-\alpha} \chi_{\{|\xi |^{\sigma} > M\}} \right) \, e^{ix\xi} d\xi .
 \end{align*}
Once again, the derivative of the function is absolutely integrable and so  
\[\Big | |\nabla |^{-s} \mathcal{T}_1(x)\Big | \lesssim_s |x|^{-1}.\]
Combining the $O(1)$ and the $O(|x|^{-1})$ bounds, \cref{HLS} follows. At this stage, while we could use the full decay $O(\langle x\rangle^{-1})$ to complete the proof of this proposition, we prefer to follow a strategy of proof that also works for the other operators where this decay is not available.
Finally, we use the Hardy-Littlewood-Sobolev inequality to prove the final bound
\begin{align}\label{usedHLS}
\norm{|\nabla |^{-s} T_t f}_{L^p_x L^{\infty}_T } & = \norm{(|\nabla |^{-s}\mathcal{T}_t) \ast f}_{L^p_x L^{\infty}_T }\leq \norm{ \sup_t ||\nabla |^{-s} \mathcal{T}_t|\ast |f|}_{L^p_x } \\
& \lesssim_s \norm{ |\cdot|^{s-1} \ast |f|}_{L^p_x}\lesssim_s \norm{f}_{L^2_x} .
\end{align}
The assumption $s=\frac{1}{2}-\frac{1}{p}>0$ is a necessary hypothesis to use the Hardy-Littlewood-Sobolev inequality (see \cite{Stein}).
\end{proof}

\begin{proposition} For $p> 2$ and $U_t$ as defined in \cref{Loperators}, we have
\[ \norm{U_t f}_{L^p_x L^{\infty}_t } \lesssim \norm{|\nabla |^s \, f}_{L^2_x} ,\]
where $s=\frac{1}{2}-\frac{1}{p}$ (and for $p=\infty$ we have $s=\frac{1}{2}$).
\end{proposition}
\begin{proof}
We define
\begin{align*}
\mathcal{U}_t (x) & := \int_{\RR} \chi_{\{t|\xi |^{\sigma} \leq M\}}\,  E_{\beta}(i^{-\beta} t^{\beta} |\xi |^{\alpha}) \, e^{ix\xi} d\xi \\
& = \int_{\RR} \chi_{\{t|\xi |^{\sigma} \leq M\}}\, \left( \sum_{k=0}^{\infty} \frac{ i^{-\beta k} t^{\beta k} |\xi |^{\alpha k}}{\Gamma (\beta k + 1)}\right) \, e^{ix\xi} d\xi \, .
\end{align*}
By rescaling, it suffices to prove 
\begin{equation}\label{goal2}
\Big | |\nabla |^{-s} \mathcal{U}_1(x)\Big | \lesssim_s |x|^{s-1} ,
\end{equation}
for $s\in [0,1)$, and then use the Hardy-Littlewood-Sobolev inequality as in \cref{usedHLS}.

In order to prove \cref{goal2} we can directly use integration by parts together with the fact that this integral is bounded for every $x$ (remember that the series converges absolutely in the support of $\chi$).
\begin{align*}
|\nabla |^{-s} \mathcal{U}_1(x) = & \int_{\RR} |\xi |^{-s} \, \chi_{\{|\xi |^{\sigma} \leq M\}}  \left( \sum_{k=0}^{\infty} \frac{ i^{-\beta k} |\xi |^{\alpha k}}{\Gamma (\beta k + 1)}\right) \, e^{ix\xi} d\xi \\
 = & \int_{|\xi |\leq |x|^{-1}} |\xi |^{-s} \, \chi_{\{|\xi |^{\sigma} \leq M\}}  \left( \sum_{k=0}^{\infty} \frac{ i^{-\beta k} |\xi |^{\alpha k}}{\Gamma (\beta k + 1)}\right) \, e^{ix\xi} d\xi \\
& + \int_{|\xi |>|x|^{-1}} |\xi |^{-s} \, \chi_{\{|\xi |^{\sigma} \leq M\}}  \left( \sum_{k=0}^{\infty} \frac{ i^{-\beta k} |\xi |^{\alpha k}}{\Gamma (\beta k + 1)}\right) \, e^{ix\xi} d\xi =I_1+I_2 .
\end{align*}
Therefore:
\begin{align*}
| I_1 (x) |& \lesssim \int_{|\xi |\leq |x|^{-1}} |\xi |^{-s} \, d\xi \sim |x|^{s-1} ,\\
| I_2 (x) |& \lesssim \Big | \int_{|\xi |> |x|^{-1}} \frac{1}{x}\frac{d}{d\xi}\left[ |\xi |^{-s} \, \chi_{\{|\xi |^{\sigma} \leq M\}}  \left( \sum_{k=0}^{\infty} \frac{ i^{-\beta k} |\xi |^{\alpha k}}{\Gamma (\beta k + 1)}\right) \right] e^{ix\xi} \, d\xi\Big | \\
& \leq \frac{1}{|x|} \int_{|\xi |> |x|^{-1}} \Big | \frac{d}{d\xi}\left[ |\xi |^{-s} \, \chi_{\{|\xi |^{\sigma} \leq M\}}  \left( \sum_{k=0}^{\infty} \frac{ i^{-\beta k} |\xi |^{\alpha k}}{\Gamma (\beta k + 1)}\right) \right] \Big | d\xi\lesssim |x|^{s-1} .
\end{align*}
The last inequality follows from the fact that the series is absolutely convergent in any compact interval in $\xi$, as well as its derivatives, 
and so the leading behaviour comes from $\frac{d}{d\xi}|\xi |^{-s} \sim |\xi |^{-s-1}$ which integrates to around $|x|^{s}$.
\end{proof}

One may prove the following using similar ideas.

\begin{proposition}\label{futureref} For $p> 2$ and $S_t$ as defined in \cref{Loperators}, we have
\[ \norm{S_t f}_{L^p_x L^{\infty}_t } \lesssim \norm{|\nabla |^s \, f}_{L^2_x} ,\]
where $s=\frac{1}{2}-\frac{1}{p}$ (and for $p=\infty$ we have $s=\frac{1}{2}$).
\end{proposition}

Now let us study what will constitute the nonlinear terms under the same norm.

\begin{proposition} For $p\geq 4$ and $G(t, \cdot )\in \mathcal{S}(\RR )$ pointwise for each $t\in [0,T]$, we have
\[\norm{ \int_0^t |\nabla |^{\sigma -\alpha} e^{-i(t-t') |\nabla |^{\sigma}} G(t',x) \, dt'}_{L^p_x L^{\infty}_T} \lesssim \norm{|\nabla |^{\sigma-\alpha +s} G}_{L^1_T L^2_x} ,\]
where $s=\frac{1}{2}-\frac{1}{p}$.
\end{proposition}
\begin{proof}
Let $f_{t'}(x):=e^{it' |\nabla |^{\sigma}} |\nabla |^{\sigma -\alpha} G(t',x)$, then by applying \cref{maxfunction} and the fact that the operator is unitary we have
\begin{align*}
\Big\lVert \int_0^t |\nabla |^{\sigma -\alpha} & e^{-i(t-t') |\nabla |^{\sigma}} G(t',x) \, dt'\Big\rVert_{L^p_x L^{\infty}_T} \\
& \leq \int_0^T \norm{ |\nabla |^{\sigma -\alpha} e^{-i(t-t') |\nabla |^{\sigma}} G(t',x)}_{L^p_x L^{\infty}_t ([0,T])} \, dt' \\
& =\int_0^T \norm{ e^{-it|\nabla |^{\sigma}} f_{t'}(x)}_{L^p_x L^{\infty}_t ([0,T])} \, dt'\\
& \lesssim \int_0^T \norm{ |\nabla |^{s} f_{t'}(x)}_{L^2_x} \, dt' = \int_0^T \norm{e^{-it' |\nabla |^{\sigma}}  |\nabla |^{\sigma -\alpha+ s} G(t',x)}_{L^2_x} \, dt' \\
& =\int_0^T \norm{ |\nabla |^{\sigma -\alpha+ s} G(t',x)}_{L^2_x} \, dt'\ . 
\end{align*}
\end{proof}

\begin{proposition}\label{futureref2} For $p>2$ and $\tilde{S}_t$ as defined in \cref{NLoperators}, we have
\[\norm{ \int_0^t \tilde{S}_{t-t'} G(t',x) \, dt'}_{L^p_x L^{\infty}_T} \lesssim \norm{|\nabla |^{\sigma-\alpha +s} G}_{L^1_T L^2_x} ,\]
for $s=\frac{1}{2}-\frac{1}{p}$.
\end{proposition}
\begin{proof}
Note that $|\nabla |^{-s-\sigma+\alpha} \tilde{S}_t=|\nabla |^{-s} S_t(x)$,
and therefore we may use \cref{futureref} for the operator $S_t$
\begin{align*}
\norm{ \int_0^t |\nabla |^{-s-\sigma+\alpha} \tilde{S}_{t-t'} G(t',x) \, dt'}_{L^p_x L^{\infty}_T} & =\norm{ \int_0^t |\nabla |^{-s} S_{t-t'} G(t',x) \, dt'}_{L^p_x L^{\infty}_T} \\
& \leq \int_0^T \norm{ |\nabla |^{-s} S_{t-t'} G(t',x)}_{L^p_x L^{\infty}_t ([t',T])} \, dt' \\
& \lesssim \int_0^T \norm{G(t',x)}_{L^2_x} \, dt' .
\end{align*}
\end{proof}

\begin{rem} We can also prove the bound 
\[\norm{ \int_0^t \tilde{S}_{t-t'} G(t',x) \, dt'}_{L^p_x L^{\infty}_T} \lesssim T^{\beta-\frac{1}{2}} \, \norm{|\nabla |^{s} G}_{L^2_{T,x}} .\]
 The main idea is to write 
 \begin{align*}
|\nabla |^{-s} \mathscr{S}_t(x):= & \int_{\RR} |\xi |^{-s+\sigma-\alpha} \chi_{\{t|\xi |^{\sigma} \leq M\}}\, e^{it|\xi |^{\sigma}+ix\xi}\, d\xi \\
= & t^{\beta -1} \int_{\RR} |\xi |^{-s} \chi_{\{t|\xi |^{\sigma} \leq M\}}\, (t^{1-\beta} |\xi |^{\sigma-\alpha})\, e^{it|\xi |^{\sigma}+ix\xi}\, d\xi ,
\end{align*}
and then we treat the piece in parentheses as a bounded function and proceed as with the proof for $\tilde{U}_t$ below.
The difference is that this second proof only works for finite $T$, whereas the one for \cref{futureref2} works for $T=\infty$ as well.
\end{rem}

\begin{proposition} For $p>2$ and $\tilde{T}_t$ as defined in \cref{NLoperators}, we have
\[\norm{ \int_0^t \tilde{T}_{t-t'} G(t',x) \, dt'}_{L^p_x L^{\infty}_T} \lesssim \norm{|\nabla |^{\sigma-\alpha +s} G}_{L^1_T L^2_x} ,\]
for $s=\frac{1}{2}-\frac{1}{p}$. Moreover ,
\[\norm{ \int_0^t \tilde{T}_{t-t'} G(t',x) \, dt'}_{L^p_x L^{\infty}_T} \lesssim T^{\beta-\frac{1}{2}} \, \norm{|\nabla |^{s} G}_{L^2_{T,x}} .\]
\end{proposition}
\begin{proof}
Both proofs are analogous to those of \cref{futureref2}.
\end{proof}

\begin{proposition} For $p>2$ and $\tilde{U}_t$ as defined in \cref{NLoperators}, we have
\[\norm{ \int_0^t \tilde{U}_{t-t'} G(t',x) \, dt'}_{L^p_x L^{\infty}_T} \lesssim T^{\beta-\frac{1}{2}} \, \norm{|\nabla |^{s} G}_{L^2_{T,x}} ,\]
for $s=\frac{1}{2}-\frac{1}{p}$.
\end{proposition}
\begin{proof}
In this case we need to be slightly more careful. Based on the definition of $\tilde{U}_t$ in \cref{NLoperators} we let
\[|\nabla |^{-s} \mathscr{U}_t(x):=\int_{\RR} t^{\beta-1} |\xi |^{-s} \chi_{\{t|\xi |^{\sigma} \leq M\}}  \left( \sum_{k=0}^{\infty} \frac{t^{\beta k} i^{-\beta k} |\xi |^{\alpha k}}{\Gamma (\beta k + \beta)}\right) \, e^{ix\xi} d\xi ,\]
where $\sigma=\frac{\alpha}{\beta}$. As with $|\nabla |^{-s} U_t$, one proves that $\Big | |\nabla |^{-s} \mathscr{U}_t(x) \Big |\lesssim_s t^{\beta -1} |x|^{s-1}$, where the implicit constant is independent of time. Then
\begin{align*}
\sup_{t\in [0,T]} \Big | \int_0^t |\nabla |^{-s} & \mathscr{U}_{t-t'} \ast G_{t'} \, dt' \Big | \leq \sup_{t\in [0,T]}  \int_0^t ||\nabla |^{-s}\mathscr{U}_{t-t'}| \ast |G_{t'}| \, dt' \\
& \lesssim \sup_{t\in [0,T]}  \int_0^t |t-t' |^{\beta -1} |x|^{s-1} \ast |G_{t'}(x)| \, dt' \\
& =|x|^{s-1} \ast \left( \sup_{t\in [0,T]}  \int_0^t |t-t' |^{\beta -1} |G_{t'}(x)| \, dt'\right) \\
& \leq |x|^{s-1} \ast \left[ \sup_{t\in [0,T]}  \left(\int_0^t |t-t' |^{2(\beta -1)}dt'\right)^{\frac{1}{2}} \norm{G(t',x)}_{L^2_{t'}([0,t])}\right] \\
& =|x|^{s-1} \ast \left(  T^{\beta -\frac{1}{2}} \norm{G(t',x)}_{L^2_T}\right) ,
\end{align*}
and we finish by using the Hardy-Littlewood-Sobolev inequality.
\end{proof}

We sum up these results in the following:

\begin{theorem}\label{maxth}
Suppose $4\leq p<\infty$ and $s=\frac{1}{2}-\frac{1}{p}$, and $\alpha>\frac{\sigma +1}{2}$. Then
\begin{align*}
\norm{e^{-it|\nabla |^{\sigma}}f}_{L^p_x L^{\infty}_T} & \lesssim \norm{|\nabla |^s f}_{L^2_x} ,\\
\norm{S_t f}_{L^p_x L^{\infty}_t } & \lesssim \norm{|\nabla |^s \, f}_{L^2_x} ,\\
\norm{T_t f}_{L^p_x L^{\infty}_t } & \lesssim \, \norm{|\nabla |^s \, f}_{L^2_x} ,\\
\norm{U_t f}_{L^p_x L^{\infty}_t } & \lesssim \norm{|\nabla |^s \, f}_{L^2_x} .
\end{align*}
Moreover,
\begin{align*}
\norm{ \int_0^t |\nabla |^{\sigma -\alpha} e^{-i(t-t') |\nabla |^{\sigma}} G(t',x) \, dt'}_{L^p_x L^{\infty}_T} & \lesssim \norm{|\nabla |^{\sigma-\alpha +s} G}_{L^1_T L^2_x} ,\\
\norm{ \int_0^t \tilde{S}_{t-t'} G(t',x) \, dt'}_{L^p_x L^{\infty}_T} & \lesssim \norm{|\nabla |^{\sigma-\alpha +s} G}_{L^1_T L^2_x} ,\\
\norm{ \int_0^t \tilde{T}_{t-t'} G(t',x) \, dt'}_{L^p_x L^{\infty}_T} & \lesssim \norm{|\nabla |^{\sigma-\alpha +s} G}_{L^1_T L^2_x} ,\\
\norm{ \int_0^t \tilde{U}_{t-t'} G(t',x) \, dt'}_{L^p_x L^{\infty}_T} & \lesssim T^{\beta-\frac{1}{2}} \, \norm{|\nabla |^{s} G}_{L^2_{T,x}} .
\end{align*}
\end{theorem}

\noindent The norm non-homogeneous version of this theorem follows from the estimate 
\[\norm{|\nabla |^{s} f}_{L^2_x} \leq  \norm{\langle \nabla \rangle^{s} f}_{L^2_x}\] 
for $s\geq 0$.

\section{Proof of \cref{maincor}}
\label{sec: easy proof}

In order to simplify our exposition, we will consider a special case of the initial value problem \cref{linear}. Namely, consider a cubic nonlinearity $p=3$, and the following specific values for the parameters
\[ \alpha=\frac{7}{4} ,\qquad \beta=\frac{7}{8},\qquad \sigma=2,\qquad \gamma=\frac{1}{2},\qquad \tilde{\gamma}=\frac{1}{4},\]
which clearly satisfy the conditions in \cref{parameterconditions}.

Based on the integral equation \cref{inteq}, we define the following operator
\begin{align}
\Phi (v)& := \int_{\RR}  \widehat{f}(\xi) \, E_{\beta}(|\xi |^{\alpha} t^{\beta} i^{-\beta}) e^{ix\cdot \xi} \, d\xi \label{fracduhamel}\\
& + i^{-\beta} \, \int_0^{t} \int_{\RR} \widehat{g}(\tau , \xi) \, (t-\tau )^{\beta -1} E_{\beta , \beta }(i^{-\beta} (t-\tau )^{\beta} |\xi |^{\alpha}) \, e^{ix\cdot \xi}\, d\xi d\tau  ,\nonumber
\end{align}
where 
\[\widehat{g}(\tau , \xi)=\int_{\RR} |v(\tau,x)|^2 v(\tau,x) e^{-ix\cdot \xi} \, dx \ .\]

We remind the reader of the notation $a-$ to denote the number $a-\varepsilon$ for $0<\varepsilon \ll 1$ small enough. Similarly, we denote by $a+$ the number $a+\varepsilon$ for $0<\varepsilon \ll 1$ small enough.

Now we define the norms
\begin{align*}
\eta_1 (v) & = \norm{\langle \nabla \rangle^{\frac{3}{4}-}\, v }_{L^{\infty}_x L^2_T} ,\\
\eta_2 (v) & = \norm{\langle \nabla \rangle^{\frac{1}{4}}\, v }_{L^{\infty}_T L^2_x} ,\\
\eta_3 (v) & = \norm{ v }_{L^{4}_x L^{\infty}_T} ,
\end{align*}
and let $\Lambda_T:= \max_{j=1,2,3}  \,\eta_j$. Then consider the space 
\[X_T:=\{v\in C([0,T],H^{\frac{1}{4}}(\RR )) \mid \Lambda_T (v)<\infty \} .\]
Our goal is to show that for small enough $T$, there exists a ball $B_R \subset X_T$ such that $\Phi : B_R \longrightarrow B_R$ is a contraction, and then apply the contraction mapping theorem.

Before we do that, we prove a useful lemma.

\begin{lemma}\label{series} For $j=1,2,3$ the following inequality holds
\begin{align*}
\eta_j (\Phi (v))  \lesssim & \ \eta_j (e^{-it|\nabla |^{\sigma}}f)+ \eta_j(S_t f)+ \eta_j(T_t^{\dagger} f)+ \eta_j(U_t f) \\
& + \eta_j \left(\int_0^t |\nabla |^{\sigma-\alpha} \, e^{-i(t-t')|\nabla |^{\sigma}} g(t',\cdot) \, dt'\right)  + \eta_j\left(\int_0^t \tilde{S}_{t-t'}g(t',\cdot) \, dt'\right)\\
& +\eta_j\left(\int_0^t \tilde{T}_{t-t'}^{\dagger} g(t',\cdot) \, dt'\right)+ \eta_j\left(\int_0^t \tilde{U}_{t-t'}g(t',\cdot) \, dt'\right).
\end{align*}
where $T_t^{\dagger}$ and $\tilde{T}_t^{\dagger}$ satisfy the same estimates as $T_t$ and $\tilde{T}_t$, respectively.
\end{lemma}
\begin{proof}
Let $u$ be the linear part of $\Phi (v)$, and remember that $\chi:=\chi_{\{t|\xi |^{\sigma} \leq M\}}$ is a smooth function such that $\chi=1$ if $t|\xi |^{\sigma} \leq M$ and $\chi=0$ if $t|\xi |^{\sigma} >2M$. Let $\widehat{P_{\chi}u}=\chi \widehat{u}$. Then we have
\begin{align*}
u & =P_{\chi}u + P_{1-\chi}u =  U_t f + P_{1-\chi}u \\
& = U_t f + e^{-it|\nabla |^{\sigma}} f - P_{\chi} e^{-it|\nabla |^{\sigma}} f + \left( P_{1-\chi}u - P_{1-\chi} e^{-it|\nabla |^{\sigma}} f\right)\\
 & = U_t f + e^{-it|\nabla |^{\sigma}} f - S_t f + P_{1-\chi}\left(u - e^{-it|\nabla |^{\sigma}} f \right).
\end{align*}
Now by \cref{Lasymp}, the Fourier multiplier corresponding to $P_{1-\chi}\left(u - e^{-it|\nabla |^{\sigma}} f\right)$ may be controlled as follows for large enough $M$
\[\left(1-\chi(t,\xi)\right) \, \Big | E_{\beta}(|\xi |^{\alpha} t^{\beta} i^{-\beta}) - e^{-it|\xi |^{\sigma}}\Big |\lesssim \left(1-\chi(t,\xi)\right)\, \frac{1}{t^{\beta} |\xi |^{\alpha}} .\]
As a result, one can check that $P_{1-\chi}\left(u - e^{-it|\nabla |^{\sigma}} f\right)$ satisfies the same estimates as $T_t$, since in every proof involving $T_t$ the absolute value was taken at some stage. The idea for the nonlinear part of $\Phi (v)$ is similar.
\end{proof}

Let us start by taking the first norm of \cref{fracduhamel}. We will use a combination of \cref{series} and \cref{thsmooth} to control the two terms that form  $\Phi (v)$. Let us remind the reader of the estimates proved in \cref{thsmooth} for the particular case of our parameters. 

\begin{align*}
\norm{\langle\nabla\rangle^{\frac{1}{2}} e^{-it|\nabla |^2} f}_{L^{\infty}_x L^2_T}& \lesssim (1+T^{\frac{1}{2}}) \norm{f}_{L^2_x} ,\\
\norm{\langle\nabla\rangle^{\frac{1}{2}-} S_t f}_{L^{\infty}_x L^2_T} & \lesssim T^{0+} (1+T^{\frac{1}{2}}) \norm{f}_{L^2_x} ,\\
\norm{\langle\nabla\rangle^{\frac{1}{2}-} T_t f}_{L^{\infty}_x L^2_T} & \lesssim T^{0+} (1+T^{\frac{1}{2}}) \norm{f}_{L^2_x} ,\\
\norm{\langle\nabla\rangle^{\frac{1}{2}-} U_t f}_{L^{\infty}_x L^2_T} & \lesssim T^{0+} (1+T^{\frac{1}{2}}) \norm{f}_{L^2_x} , \\
\norm{\langle\nabla\rangle^{\frac{1}{4}} \int_0^t |\nabla|^{\frac{1}{4}} e^{-i(t-t')|\nabla |^{2}} G(t',x) dt'}_{L^{\infty}_x L^2_T}& \lesssim (1+T^{\frac{1}{2}}) \norm{G}_{L^1_T L^2_x} ,\\
\norm{\langle\nabla\rangle^{\frac{1}{4}-} \int_0^t \tilde{S}_{t-t'} G(t',x) dt'}_{L^{\infty}_x L^2_T} & \lesssim (T^{0+} +T^{\frac{1}{2}}) \norm{G}_{L^1_T L^2_x} ,\\
\norm{\langle\nabla\rangle^{\frac{1}{4}-} \int_0^t \tilde{T}_{t-t'} G(t',x) dt'}_{L^{\infty}_x L^2_T} & \lesssim  T^{0+} \, \norm{G}_{L^1_T L^2_x} ,\\
\norm{\langle\nabla\rangle^{\frac{1}{4}-} \int_0^t \tilde{U}_{t-t'} G(t',x) dt'}_{L^{\infty}_x L^2_T} & \lesssim (T^{0+} +T^{\frac{1}{2}}) \norm{G}_{L^1_T L^2_x} .
\end{align*}

These estimates above, together with \cref{series} allow us to control \cref{fracduhamel} in terms of $\eta_1$ as follows:
\begin{equation}\label{eta1}
\eta_1 (\Phi (v))\lesssim (1+ T^{\frac{1}{2}})\, (1+ T^{0+})\, \norm{\langle \nabla \rangle^{\frac{1}{4}} f}_{L^2_x} + (1+T^{0+} +T^{\frac{1}{2}})\, \norm{\langle \nabla \rangle^{\frac{1}{2}} (|v|^2 v)}_{L^1_T L^2_x} .
\end{equation}
Note the gain in derivatives in this norm, almost $\frac{1}{2}$ for the linear part and almost $\frac{1}{4}$ in the nonlinear part.

Now we look at the second norm of \cref{fracduhamel}. \Cref{L21,L22,L23} combined with \cref{series} yield
\begin{equation}\label{eta2}
\eta_2 (\Phi (v))\lesssim \norm{\langle \nabla \rangle^{\frac{1}{4}} f}_{L^2_x} + T^{\frac{3}{8}}\, (1+ T^{\frac{1}{8}}) \, \norm{ \langle\nabla\rangle^{\frac{1}{2}} (|v|^2 v)}_{L^2_{T,x}} .
\end{equation}
Note that in this case we clearly lose $\frac{1}{4}$ derivatives in the nonlinear part. 

Finally, we take the third norm, which we control thanks to \cref{maxth} and \cref{series}.
\begin{equation}\label{eta3}
\eta_3 (\Phi (v))\lesssim \norm{\langle \nabla \rangle^{\frac{1}{4}} f}_{L^2_x}+ T^{\frac{3}{8}}\, \norm{ \langle\nabla\rangle^{\frac{1}{2}} (|v|^2 v)}_{L^2_{T,x}} .
\end{equation}
Let us highlight once again the loss of $\frac{1}{4}$ derivatives in the linear part and $\frac{1}{2}$ in the nonlinear part.

Therefore, we can simultaneously estimate \cref{eta1,eta2,eta3} by controlling the quantity $\norm{ \langle\nabla\rangle^{\frac{1}{2}} (|v|^2 v)}_{L^2_{T,x}}$ in terms of $\eta_1 , \eta_2$ and $\eta_3$.
To this end, we introduce the following:

\begin{lemma}\label{AvoidChainRule} When $p$ is an odd integer, we have that
\[\norm{ \langle\nabla\rangle^{s} (|v|^{p-1} v)}_{L^2_{T,x} }\lesssim \norm{ \langle\nabla\rangle^{s}v}_{L^{\infty}_x L^2_T}\, \norm{v}_{L^{2(p-1)}_x L^{\infty}_T}^{p-1} .\]
\end{lemma}
\begin{proof}
When $p=2k+1$ is an odd integer, we have $|u|^{p-1}u=u^{k+1} \bar{u}^k$. Then by the Plancherel theorem,
\[\norm{ \langle\nabla\rangle^{s} (|v|^{p-1} v)}_{L^2_x}\lesssim \norm{(\langle\nabla \rangle^{s} v) |v|^{p-1}}_{L^2_x}+ \norm{ v\, \langle\nabla\rangle^{s} (|v|^{p-1})}_{L^2_x} .\]
Now we iterate the argument for the last term to argue that the derivative can only hit $v$ or $\bar{v}$. In other words,
\[\norm{ \langle\nabla\rangle^{s} (|v|^{p-1} v)}_{L^2_x}\lesssim \norm{(\langle\nabla \rangle^{s} v) |v|^{p-1}}_{L^2_x}+\norm{(\langle\nabla \rangle^{s} \bar{v} ) \, v^{k+1}\, \bar{v}^{k-1} }_{L^2_x} .\]
After taking the $L^2_T$ norm, and interchaging the order of integration in $T$ and $x$, we can use the H\"older inequality to control each term in terms of 
\[\norm{\langle\nabla\rangle^{s}v}_{L^{\infty}_x L^2_T}\, \norm{v}_{L^{2(p-1)}_x L^{\infty}_T}^{p-1}.\]
\end{proof}

 Using \cref{AvoidChainRule} for the case $p=3$, we find that
\[\norm{ \langle\nabla\rangle^{\frac{1}{2}} (|v|^2 v)}_{L^2_T L^2_x}\lesssim \eta_1 (v) \, \eta_3 (v)^2 \, .\]

By assuming that $T\leq 1$, \cref{eta1,eta2,eta3} can be rewritten as
\[\Lambda_T (\Phi (v))\lesssim \norm{f}_{H^{1/4}_x}+ T^{\frac{3}{8}} \, \eta_1 (v) \, \eta_3 (v)^2 \lesssim \norm{f}_{H^{1/4}_x}+ T^{\frac{3}{8}} \, \Lambda_T (v)^3 .\]

Pick $B_R = \{ v\in X_T \mid \Lambda_T (v) < R \}$ and $f$ such that $\norm{f}_{H^{1/4}_x}\lesssim \frac{R}{2}$. Then we have that $\Phi : B_R \longrightarrow B_R$ as long as $T^{\frac{3}{8}} R^3 \lesssim \frac{R}{2}$, which happens for $T$ small enough. We still must prove that $\Phi (v)$ is actually in $C ([0,T],H^{\frac{1}{4}}(\RR ) )$, but we do that for the general case in \cref{cont}, which can be found in \cref{sec: appB} below.

Now we prove that $\Phi$ is a contraction. By using the same ideas as in \cref{eta1,eta2,eta3} together with $T\leq 1$ we quickly find
\[\Lambda_T (\Phi (v)- \Phi (u))\lesssim T^{\frac{3}{8}} \, \norm{ \langle\nabla\rangle^{\frac{1}{2}} (|v|^2 v - |u|^2 u)}_{L^2_{T,x}} .\]
In order to deal with this last term, we use the following estimate based on the fundamental theorem of calculus:
\begin{equation*}
\norm{|v|^2 v - |u|^2 u}_{L^2_T H^{1/2}_x} 
\lesssim \int_0^1 \norm{ |v + \lambda (u-v) |^{2} (u-v)}_{L^2_T H^{1/2}_x} \, d\lambda .\\
\end{equation*}
Then a combination of \cref{AvoidChainRule} and the H\"older inequality yields
\[\Lambda_T (\Phi (v)- \Phi (u))\lesssim T^{\frac{3}{8}}\, R^2 \,\Lambda_T (v-u) ,\]
so that by making $T$ smaller (if necessary), we can arrange $C T^{\frac{3}{8}}\, R^2<1$. Therefore $\Phi$ is a contraction on the ball $B_R = \{ v \mid \Lambda_T (v) < R\sim \norm{f}_{H^{1/4}_x} \}$, which shows that $\Phi$ has the desired contraction property.

\section{Ill-posedness}
\label{sec: ill-posedness}

\subsection{Setup}

We consider the initial value problem:
\begin{equation}\label{perturbed}
\left\{
    \begin{array}{ll}
    i^{\beta}\partial_t^{\beta} u_{\epsilon} & = (-\Delta_x)^{\frac{\alpha}{2}}u_{\epsilon} + \mu |u_{\epsilon}|^{p-1}u_{\epsilon},\\
    u_{\epsilon}|_{t=0} & =\epsilon \, f ,
    \end{array}
\right.
\end{equation}
with $\alpha\in (0,2)$, $\beta\in (0,1)$, $\sigma:=\frac{\alpha}{\beta}$, $\mu=\pm 1$, small $\epsilon$ and smooth initial data. From now on we will consider the case $\mu=1$, since both cases are handled similarly.

Note that the initial data-to-solution map from $H^s(\RR)$ to $C_t([0,T],H^s_x(\RR))$ of \cref{linear} will be $C^{k}$ if and only if the map that sends 
\[ f\in H^s(\RR)\mapsto (\partial^{k}_{\epsilon}u_{\epsilon} )\Big |_{\epsilon=0}\in C_t([0,T],H^s_x(\RR))\]
is continuous.

More precisely, suppose that $u_{\epsilon}$ admits an asymptotic expansion of type 
\[ u_{\epsilon}= \epsilon u_1+\epsilon^2 u_2 + \ldots \]
By plugging this into \cref{perturbed} and matching the coefficients of the powers of $\epsilon$, we find that $u_1$ is simply the solution to the linear problem with initial datum $f$, i.e.
\begin{equation}\label{order1}
\left\{
    \begin{array}{ll}
    i^{\beta}\partial_t^{\beta} u_1 & = (-\Delta_x)^{\frac{\alpha}{2}}u_1\\
    u_1|_{t=0} & =f.
    \end{array}
\right.
\end{equation}
In the same way, $u_2=\ldots =u_{p-1}=0$, since they solve:
\[\left\{
    \begin{array}{ll}
    i^{\beta}\partial_t^{\beta} u_j & = (-\Delta_x)^{\frac{\alpha}{2}}u_j\\
    u_j|_{t=0} & =0.
    \end{array}
\right.\]
Finally $u_p$ will satisfy:
\begin{equation}\label{orderp}
\left\{
    \begin{array}{ll}
    i^{\beta}\partial_t^{\beta} u_p & = (-\Delta_x)^{\frac{\alpha}{2}}u_p + |u_1|^{p-1}u_1,\\
    u_p|_{t=0} & =0.
    \end{array}
\right.
\end{equation}
By the fractional Duhamel principle, and denoting the nonlinearity $g(u):=|u|^{p-1}u$ as usual, we obtain that 
\[ u_p(t)=i^{\beta} \int_0^t (t-s)^{\beta-1} \left(E_{\beta,\beta}((t-s)^{\beta} i^{-\beta} |\cdot |^{\alpha}) \widehat{g(u_1)}(s,\cdot)\right)^{\vee} \, ds.\]

Suppose that we have a solution $u_j$ for each of these IVPs we mentioned before (only $j\in\{1,p\}$ are nonzero), which exists in some common time interval $[0,T^{\ast}]$, for some small enough $T^{\ast}$.
Our goal is to show that the operator that maps $f$ to $u_p$ is not continuous from $H^s_x$ to $C([0,T^{\ast}],H^s_x)$ for $s<s_c=\frac{1}{2}-\frac{\alpha}{p-1}$, however small $T^{\ast}$ is. 

We consider the following family of initial data:
\[ \widehat{f_N}(\zeta):= N^{\varepsilon-s} \,\chi_{[N-\frac{1}{N^{2\varepsilon}},N]}(\zeta)\]
for some $\varepsilon>-\frac{1}{2}$ which will be specified later. Here $\chi_A$ denotes the characteristic function on the set $A$. Note that the coefficient has been chosen so that 
\[ \norm{f_N}_{H^s}=1+O(N^{-1-2\varepsilon})\]
for large $N$.
\subsection{Computations}

Now we approximate $u_1$, the solution of the linear problem \cref{order1}. Recall that as $t^{\beta} |\zeta|^{\alpha}\rightarrow\infty$, we have the asymptotics:
\[ E_{\beta}(t^{\beta} i^{-\beta} |\zeta |^{\alpha})= \frac{1}{\beta}\, e^{-it|\zeta|^{\sigma}}+c\, t^{-\beta}|\zeta|^{-\alpha} + \mbox{l.o.t} ,\]
where $\mbox{l.o.t}$ denotes lower order terms with respect to the product $t^{-\beta}|\zeta|^{-\alpha}$, see \cref{Lasymp}. Note that we are working with $\zeta$ in the interval $[N-\frac{1}{N^{2\varepsilon}},N]$, and therefore $t^{-\beta}|\zeta|^{-\alpha}$ and $t^{-\beta} N^{-\alpha}$ are comparable.

\begin{proposition}\label{u1prop} Suppose that $\varepsilon>\frac{\sigma-1}{2}$. Then the following approximation is valid for large $t^{\beta}N^{\alpha}$:
\begin{equation}\label{u1}
u_1(t,x)= \left(N^{\varepsilon-s} \, e^{-itN^{\sigma}}+N^{\varepsilon-s}\, t^{-\beta}\, N^{-\alpha}+\mbox{l.o.t}\right) \left(\frac{1-e^{-ixN^{-2\varepsilon}}}{ix}\right)\, e^{ixN}.
\end{equation}
\end{proposition}
\begin{proof}
For large $t^{\beta}N^{\alpha}$, we may use the asymptotics of $E_{\beta}$ to write:
\begin{align*}
u_1(t,x) = & \int_{\RR} E_{\beta}(t^{\beta}i^{-\beta}|\zeta|^{\alpha}) \widehat{f_N}(\zeta) e^{ix\zeta}\,d\zeta
 \\
  = & \, N^{\varepsilon-s} \,e^{-itN^{\sigma}}\,\int_{\zeta\in [N-\frac{1}{N^{2\varepsilon}},N]} e^{-it(|\zeta|^{\sigma}-N^{\sigma})} e^{ix\zeta}\,d\zeta\\
 & +N^{\varepsilon-s}\, \int_{\zeta\in [N-\frac{1}{N^{2\varepsilon}},N]} t^{-\beta}|\zeta|^{-\alpha}\, e^{ix\zeta}\,d\zeta  + \mbox{l.o.t.}
\end{align*}

By the binomial approximation and the Taylor series for the exponential, we have:
 \begin{equation}\label{futurecond2}
 e^{-it(|\zeta|^{\sigma}-N^{\sigma})}=1+O(t\,N^{-2\varepsilon+\sigma-1}).
 \end{equation} 
 
 At this stage, we add the condition $\frac{\sigma-1}{2}<\varepsilon$ for the error to be small (and decreasing with $N$). Note that this allows the possibility of $\varepsilon$ being negative if $\sigma$ falls below 1. 

 Finally, we can compute the first integral:
\begin{align*}
  \int_{\zeta\in [N-\frac{1}{N^{2\varepsilon}},N]} e^{-it(|\zeta|^{\sigma}-N^{\sigma})} e^{ix\zeta}\,d\zeta & = \left(1+ O(tN^{-2\varepsilon+\sigma-1})\right) \int_{\zeta\in [N-\frac{1}{N^{2\varepsilon}},N]} e^{ix\zeta} d\zeta \\ 
  &=e^{ixN}\, \left(1+ O(tN^{-2\varepsilon+\sigma-1})\right)\, \left(\frac{1-e^{-ixN^{-2\varepsilon}}}{ix}\right) .
\end{align*}

Using the fact that $|\zeta|^{-\alpha}-N^{-\alpha}=O(N^{-\alpha-1-2\varepsilon})$, we obtain for the second integral:
\[ N^{\varepsilon-s}\, \int_{\zeta\in [N-\frac{1}{N^{2\varepsilon}},N]} t^{-\beta}|\zeta|^{-\alpha}\, e^{ix\zeta}\,d\zeta=
N^{\varepsilon-s}t^{-\beta}N^{-\alpha} e^{ixN}\,\left(\frac{1-e^{-ixN^{-2\varepsilon}}}{ix}\right)+ \mbox{l.o.t},\]

Combining these two, we obtain \cref{u1}.
\end{proof}

We may now compute the nonlinearity (up to a constant). From now on, we use the notation $x^p:=|x|^{p-1} x$ for simplicity.
\[  g(u_1)(t,x) = \left( N^{p(\varepsilon-s)} \, e^{-itN^{\sigma}}+N^{p(\varepsilon-s)}\, t^{-p\beta}\, N^{-p\alpha}+\mbox{l.o.t}\right) \, \left(\frac{1-e^{-ixN^{-2\varepsilon}}}{ix}\right)^p \, e^{ixN}.  \]

We also compute its Fourier transform:
\begin{equation}\label{FTgu1}
  \widehat{g(u_1)}(t,\xi) = \left( N^{p(\varepsilon-s)} \, e^{-itN^{\sigma}}+N^{p(\varepsilon-s)}\, t^{-p\beta}\, N^{-p\alpha}+\mbox{l.o.t}\right)\, h_{p,N}(\xi),
 \end{equation}
 where
 \begin{equation}
  h_{p,N}(\xi):=\int_{\RR} \left(\frac{1-e^{-ixN^{-2\varepsilon}}}{ix}\right)^p \, e^{ix(N-\xi)}\, dx.
 \end{equation}

Now we may approximate $u_p$, the solution of \cref{orderp}, by the fractional Duhamel formula (see \cref{inteq}). There are three important intervals when doing this: $J_1=[0,CN^{-\sigma}]$, $J_2=[CN^{-\sigma},T-C|\xi|^{-\sigma}]$ and $J_3=[T-C|\xi|^{-\sigma},T]$, where $C$ is some large constant so that the asymptotics we have mentioned are valid.
\begin{align}
\widehat{u_{p,N}}(T,\xi) & = \sum_{j=1}^{3} \, \int_{J_j} (T-t)^{\beta-1} \, E_{\beta,\beta}((T-t)^{\beta}i^{-\beta} |\xi|^{\alpha}) \widehat{g(u_1)}(t,\xi)\, dt \label{up}\\
& =: I+I\!I+I\!I\!I,\nonumber
\end{align}
where $E_{\beta,\beta}$ is the generalized Mittag-Leffler function given in \cref{genML}.

The contributions of integrating over the intervals $J_1$ and $J_3$ will later be shown to produce lower order terms, and thus do not affect the leading term of these computations, which is given by the integration over $J_2$.
Note that the fact that $J_2$ is nonempty depends on 
\begin{equation}\label{futurecond1}
T>C|\xi|^{-\sigma}+ CN^{-\sigma}
\end{equation} 
for some large constant $C$. We will see later that our choice of $\xi$ (proportional to $N$) will allow $T$ (and the lifespan of $u_p$) to be as small as needed despite this condition.

We are now ready to give a sketch of what the main argument will be. If the initial data-to-solution map were $C^p$, we should have a bound of the type
\[ \norm{f}_{H^s} \gtrsim \norm{u_p}_{C_t([0,T^{\ast}], H^s_x(\RR))} \geq \norm{u_p(T)}_{H^s_x(\RR)}, \]
valid for every initial data $f$ and every $T\leq T^{\ast}$. When plugging in $f_N$, we obtain
\begin{equation}\label{mainargument}
 1\gtrsim\norm{\langle\cdot\rangle^{s}\widehat{u_{p,N}}(T)}_{L^2_{\xi}(\RR)}\geq \norm{\langle\cdot\rangle^{s}\widehat{u_{p,N}}(T)}_{L^2_{\xi}(\mathcal{I}_N)},
 \end{equation}
where $\mathcal{I}_N$ is some interval of $\xi$ that we may choose to optimize our results. One would like to take $\mathcal{I}_N$ to be $\RR$, but the approximations we will develop for $u_{p,N}$ will not be valid nor relevant on such a large set.

We claim that the optimal choice is essentially the following:
\begin{equation}\label{optimalinterval}
\mathcal{I}_N:=[N-N^{-2\delta},N]
\end{equation}
for any $\delta=\varepsilon+$. This choice will be justified in \cref{hpNprop,optimalityofINp} below. 

Note also that for $\xi\in \mathcal{I}_N$, condition \cref{futurecond1} is satisfied for any lifespan $T$, no matter how small, as long as $N\geq N_0 (T)\sim T^{-\frac{1}{\sigma}}$. 

We now present the asymptotic behavior of $u_{p,N}$ given in \cref{up}.

\begin{proposition}\label{leadingorder} For $\xi\in\mathcal{I}_N$ and all $N\geq N_0 (T)\sim T^{-\frac{1}{\sigma}}$, the following approximation holds:
\[
\widehat{u_{p,N}}(T,\xi)=N^{p(\varepsilon-s)}\, h_{p,N}(\xi)\, e^{-iT|\xi|^{\sigma}}\,N^{\sigma-\alpha}\, \, T + \mbox{l.o.t.}
 \]
\end{proposition}
\begin{proof} We divide the proof in several steps where we estimate the terms $I$, $I\!I$ and $I\!I\!I$ in \cref{up}.

{\bf Step 1.} We start by estimating $I\!I$ given by intergration over the interval $J_2$ defined in \cref{up}. For large $(T-t)|\xi|^{\sigma}$, we may combine the asymptotics of $E_{\beta,\beta}$ in \cref{NLasymp}
with \cref{FTgu1} to write:
\begin{align*}
I\!I  = & N^{p(\varepsilon-s)}\,h_{p,N}(\xi)\, e^{-iT|\xi|^{\sigma}}\,|\xi|^{\sigma-\alpha}\, \int_{J_2}   e^{it (|\xi|^{\sigma}-N^{\sigma})} \, dt  \\
&+ N^{p(\varepsilon-s)}\, h_{p,N}(\xi) \, \int_{J_2} |\xi|^{-2\alpha}(T-t)^{-1-\beta}\, e^{-itN^{\sigma}} \, dt \\
& + N^{p(\varepsilon-s)}\, h_{p,N}(\xi)\,\int_{J_2} |\xi|^{\sigma-\alpha} e^{-i(T-t)|\xi|^{\sigma}}\, t^{-p\beta}\,N^{-p\alpha} \, dt\\
& + N^{p(\varepsilon-s)}\, h_{p,N}(\xi)\,\int_{J_2} |\xi|^{-2\alpha}(T-t)^{-1-\beta}) \, t^{-p\beta}\,N^{-p\alpha} \, dt +\mbox{l.o.t.}\\
= & I\!I_1+I\!I_2+I\!I_3+I\!I_4 +\mbox{l.o.t.}
\end{align*}

We claim that the top order behavior is given by $I\!I_1$. Note that
\begin{align}\label{Nandxi} 
|\xi|^{\sigma-\alpha}\, \int_{J_2}   e^{it (|\xi|^{\sigma}-N^{\sigma})} \, dt = & N^{\sigma-\alpha}\, \int_{J_2} \, dt + N^{\sigma-\alpha} \, \int_{J_2} O(tN^{-2\delta + \sigma-1}) \, dt\\
=&  N^{\sigma-\alpha} \,T+\mbox{l.o.t.}
\end{align}

One can show that the integral in $I\!I_2$ satisfies the following bound:
\[ \Big | \int_{J_2} c |\xi|^{-2\alpha} (T-t)^{-1-\beta}\, e^{-itN^{\sigma}} \, dt\Big |\leq |\xi|^{-\alpha}.\]
Using the fact that $\xi\sim N$ (since we are in $\mathcal{I}_N$), it follows that $I\!I_2$ is a lower order term with respect to $I\!I_1$.

$I\!I_3$ is controlled as follows:
\[
\Big |\int_{J_2} (|\xi|^{\sigma-\alpha} e^{-i(T-t)|\xi|^{\sigma}}t^{-p\beta}\,N^{-p\alpha} \, dt \Big |\lesssim 
|\xi|^{\sigma-\alpha}\,N^{-p\alpha}\left( T^{1-p\beta} - |\xi|^{-\sigma+p\alpha}\right),
\]
For $\xi\in\mathcal{I}_N$ and fixed $T$, this is a lower order term with respect to $I\!I_1$. 

$I\!I_4$ is handled in the same way.

{\bf Step 2.} Now we show that the contribution of $I\!I\!I$ is negligible.
In $J_3$, the asymptotic formula for $E_{\beta,\beta}$ is not valid anymore, but we still have the following representation thanks to \cref{FTgu1}:
\[ 
I\!I\!I = h_{p,N}(\xi) \, \int_{J_3} (T-t)^{\beta-1} \, E_{\beta,\beta}((T-t)^{\beta}i^{-\beta} |\xi|^{\alpha}) \left(N^{p(\varepsilon-s)}\, e^{-itN^{\sigma}}+ \mbox{l.o.t}\right) \, dt.
\]
Because we are in $J_3$, we have that $(T-t)^{\beta}|\xi|^{\alpha}$ is bounded above by a constant. Since the function $E_{\beta,\beta}$ is continuous, it is also bounded in such an interval and therefore we have the estimate:
\begin{align*}
|I\!I\!I| & \leq  |h_{p,N}(\xi)|\, \int_{J_3} (T-t)^{\beta-1} \left( N^{p(\varepsilon-s)}+ \mbox{l.o.t}\right) \, dt\\
& \leq  |h_{p,N}(\xi)|\,\left( N^{p(\varepsilon-s)}+ \mbox{l.o.t}\right) \, \int_0^{C|\xi|^{-\sigma}} t^{\beta-1} \, dt \\
& \leq  |h_{p,N}(\xi)|\,\left( N^{p(\varepsilon-s)}+ \mbox{l.o.t}\right) \, |\xi|^{-\alpha}.
\end{align*}
Then since $\xi\in\mathcal{I}_N$, the term $|\xi|^{-\alpha}$ has order $N^{-\alpha}$, making the contribution of $I\!I\!I$ smaller than that of $I\!I$.

{\bf Step 3.} We finally deal with the contribution of $I$. In $J_1$ the asymptotics for $E_{\beta}$ are not valid and thus the formula we had for $\widehat{g(u_1)}$, namely \cref{FTgu1}, will not work. However, we can compute these elements again:
\begin{align*}
u_1(t,x) & = \int_{\RR} E_{\beta}(t^{\beta}i^{-\beta}|\zeta|^{\alpha}) \widehat{f_N}(\zeta) e^{ix\zeta}\,d\zeta\\
& = N^{\varepsilon-s} \,\left(E_{\beta}(t^{\beta}i^{-\beta}N^{\alpha})+O(N^{-2\varepsilon-1})\right)\, \int_{\zeta\in [N-\frac{1}{N^{2\varepsilon}},N]} e^{ix\zeta}\,d\zeta\\
 &= N^{\varepsilon-s} \,\left(E_{\beta}(t^{\beta}i^{-\beta}N^{\alpha})+O(N^{-2\varepsilon-1})\right)\, e^{ixN} \left(\frac{1-e^{ixN^{-2\varepsilon}}}{ix}\right).
 \end{align*}
after using the Taylor expansion $E_{\beta}(t^{\beta}i^{-\beta}|\zeta|^{\alpha})=E_{\beta}(t^{\beta}i^{-\beta}N^{\alpha})+ O(t^{\beta}N^{-2\varepsilon+\alpha-1})$ together with the fact that $t^{\beta}\lesssim N^{-\alpha}$ in $J_1$.

Therefore, 
\begin{equation*}
g(u_1)(t,x) = N^{p(\varepsilon-s)} \,\left(E_{\beta}(t^{\beta}i^{-\beta}N^{\alpha})^p+O(N^{-2\varepsilon-1})\right)\,e^{ixN} \, \left(\frac{1-e^{ixN^{-2\varepsilon}}}{ix}\right)^p
\end{equation*}
Remember that $E_{\beta}(t^{\beta}i^{-\beta}N^{\alpha})$ is bounded in $J_1$, and so it does not increase the order of the error.

Finally, 
\begin{equation*}
\widehat{g(u_1)}(t,\xi)= N^{p(\varepsilon-s)} \,\left(E_{\beta}(t^{\beta}i^{-\beta}N^{\alpha})^p+O(N^{-2\varepsilon-1})\right)\, h_{p,N}(\xi).
\end{equation*}

All in all,
\begin{multline*}
I= N^{p(\varepsilon-s)}\, h_{p,N}(\xi)\, \int_{J_1} (T-t)^{\beta-1} \, E_{\beta,\beta}((T-t)^{\beta}i^{-\beta} |\xi|^{\alpha}) \\
\left[E_{\beta}(t^{\beta}i^{-\beta}N^{\alpha})^p+O(N^{-2\varepsilon-1})\right]\, dt.
\end{multline*}

Combining \cref{NLasymp} with the fact that $E_{\beta}(t^{\beta}i^{-\beta}N^{\alpha})$ is bounded, and recalling that $\xi\in\mathcal{I}_N$, we obtain:
\begin{align*} |I|& \lesssim N^{p(\varepsilon-s)}\, |h_{p,N}(\xi)|\, \int_0^{CN^{-\sigma}} |\xi|^{\sigma-\alpha}\, \left(1+O(N^{-2\varepsilon-1})\right) \, dt\\
& =N^{p(\varepsilon-s)}\, |h_{p,N}(\xi)|\, N^{\sigma-\alpha} \, \left(1+O(N^{-2\varepsilon-1})\right)\, t\,\Big |_{0}^{CN^{-\sigma}}\\
&= N^{p(\varepsilon-s)}\, |h_{p,N}(\xi)|\,N^{-\alpha} \,\left(1+O(N^{-2\varepsilon-1})\right).
\end{align*}
Once again, this is of lower order than the contribution of $I\! I$.
\end{proof}

\subsection{The function $h_{p,N}$}

Our goal now is to understand the behavior of $h_{p,N}(\xi)$ with respect to $N$, which will motivate our choice of interval $\mathcal{I}_N$ in \cref{optimalinterval}.

\begin{proposition}\label{hpNprop} Suppose that $p>1$ and $\delta=\varepsilon+$, and $\xi\in \mathcal{I}_N$. Then
\begin{equation}\label{hpN}
h_{p,N}(\xi)=c_p \, N^{2\varepsilon (1-p)} + o(N^{2\varepsilon (1-p)}),
\end{equation}
uniformly in $\xi\in\mathcal{I}_N$.
\end{proposition}
\begin{proof}
After a change of variables:
\begin{align*}
h_{p,N}(\xi) & =\int_{\RR} \left(\frac{1-e^{ixN^{-2\varepsilon}}}{ix}\right)^p \, e^{ix(N-\xi)} \, dx\\
& =N^{2\varepsilon(1-p)}\, \int_{\RR} \left(\frac{1-e^{ix}}{ix}\right)^p \, e^{ixN^{2\varepsilon}(N-\xi)} \, dx.
\end{align*}

Remember that $e^{ixN^{2\varepsilon} (N-\xi)}=1+O(xN^{2(\varepsilon-\delta)})$ as long as $x$ is small enough. We divide the integral into two pieces to exploit this fact. The first piece is estimated as follows:
\begin{align*}
\int_{|x|<N^{\delta-\varepsilon}} & \left(\frac{1-e^{ix}}{ix}\right)^p \, e^{ixN^{2\varepsilon}(N-\xi)} \, dx \\
 = &\int_{|x|<N^{\delta-\varepsilon}} \left(\frac{1-e^{ix}}{ix}\right)^p \, dx+
 \int_{|x|<N^{\delta-\varepsilon}} \left(\frac{1-e^{ix}}{ix}\right)^p \, O(xN^{2\varepsilon-2\delta}) \, dx\\
  = & \int_{\RR} \left(\frac{1-e^{ix}}{ix}\right)^p \, dx - \int_{N^{\delta-\varepsilon}}^{\infty} O(|x|^{-p}) \, dx \\
 & + O(N^{\varepsilon-\delta}) \,\int_{|x|<N^{\delta-\varepsilon}} \left(\frac{1-e^{ix}}{ix}\right)^p \, dx= c_p + O(N^{(\delta-\varepsilon)(1-p)}) + O(N^{\varepsilon-\delta}).
 \end{align*}
The second piece gives:
\begin{equation*}
\int_{|x|>N^{\delta-\varepsilon}} \left(\frac{1-e^{ix}}{ix}\right)^p \, e^{ixN^{2\varepsilon}(N-\xi)} \, dx =
\int_{|x|>N^{\delta-\varepsilon}} O(|x|^{-p})\,dx=O(N^{(\delta-\varepsilon)(1-p)}),
\end{equation*}
which concludes the proof.
\end{proof}

Finally, we show that $\mathcal{I}_N$ defined in \cref{optimalinterval} is indeed the best choice of interval for the purpose of maximizing $h_{p,N}$ with respect to $N$. 

\begin{proposition}\label{optimalityofINp} Outside an interval comparable (in width and center) to $[N-N^{-2\varepsilon},N]$, the function $h_{p,N}(\xi)$ will display an exponential decay in $N$. This means that our choice of interval $\mathcal{I}_N=[N-N^{-2\delta},N]$ for $\delta=\varepsilon+$ is essentially optimal.
\end{proposition}
\begin{proof}
Define the function:
\begin{equation}
F(\lambda):=\int_{\RR} \left(\frac{1-e^{-ix}}{ix}\right)^p \, e^{ix\lambda} \, dx.
\end{equation}

The change of variables $x\mapsto -x$ shows that $F(\lambda)=F(1-\lambda)$. Suppose for now that $\lambda\notin [0,1]$, then the fact that $F$ is symmetric around $\frac{1}{2}$ implies we can assume that $\lambda<0$ without loss of generality.

Now consider the extension of the integrand to the complex plane: 
\begin{equation}
H(z):= \left(\frac{1-e^{-iz}}{iz}\right)^p \, e^{iz\lambda}.
\end{equation}
We can choose a rectangular shape that goes from $-R$ to $R$, then from $R$ to $R-i$, then $R-i$ to $-R-i$ and finally from $-R-i$ to $-R$. We will label each of the sides of this rectangle $L_1, \ldots, L_4$ respectively.
Then by the Cauchy integral theorem:
\[ F(\lambda)=\lim_{R\rightarrow\infty} \int_{L_1} H(z)\,dz= -\lim_{R\rightarrow\infty} \sum_{j=2}^4 \int_{L_j} H(z)\,dz,\]
assuming that these limits exist, as we will show soon. It is easy to check that as $R\rightarrow\infty$, the integrals over $L_2$ and $L_4$ vanish. As $R\rightarrow\infty$, the integral over $L_3$ will yield a term bounded by the following:
\[ e^{\lambda}\, \int_{-\infty}^{\infty} |ix+1|^{-p} \, dx=c_p e^{\lambda}.\]

Remember we are working under the assumption that $\lambda=\lambda(N,\xi)=N^{2\varepsilon}(N-\xi)<0$.
Suppose that we consider $\xi\in [a(N),b(N)]$. For a fixed $\xi$ in such an interval, unless $\xi=N+O(N^{-2\varepsilon})$, the limit $\lim_{N\rightarrow\infty} \lambda(N,\xi)$ will be $-\infty$, which will produce exponential decay.
 Consequently, the only $\xi$ that matter are those that satisfy $\xi=N+O(N^{-2\varepsilon})$. In particular, $a(N)$ and $b(N)$ admit such a representation. Since $a(N)\geq N$, the best we can do to maximize the interval is to take $a(N)=N$. Then by the above, $b(N)=N+cN^{-2\varepsilon}+ o(N^{-2\varepsilon})$ for some positive $c$.

Thus this interval is essentially comparable to $\mathcal{I}_N$ in \cref{optimalinterval}, and it is produced by taking only those $\lambda$ near zero. Using the symmetry of $F$ one can do the same for $\lambda$ near 1, and get a similar interval, which again does not improve the result.

Finally, we also assumed that $\lambda\notin [0,1]$ when we carried out this argument. That means that the only way in which we could potentially improve our interval is by including these numbers. However, these correspond to those $\xi$ satisfying 
\[0<N^{2\varepsilon}(N-\xi)<1 \quad \Leftrightarrow\quad N-N^{-2\varepsilon}<\xi<N.\]
Once again, these produce intervals for $\xi$ comparable to those we have already considered.
\end{proof}

\begin{rem} Note that the choice of $\mathcal{I}_N$ is not only optimal for $h_{p,N}$, as proved by \cref{optimalityofINp}, but it is also optimal for the asymptotic lower bound on the $H^s$-norm of $u_{p,N}$. Indeed, note that the Taylor approximations we have used in \cref{leadingorder} would be useless unless $\xi=N+o(N)$ (see for instance \cref{Nandxi}). Therefore, $\xi$ depends polynomially on $N$ in our approach. By carefully checking the proof of \cref{leadingorder}, we can see that all the lower order terms grow at most polynomially in $\xi$ and $N$, and thus polynomially in $N$. By \cref{optimalityofINp}, the function $h_{p,N}$ will decay exponentially outside of $\mathcal{I}_N$, and so even when multiplied by any of these polynomial terms in $N$ that appear in \cref{leadingorder}, it will not improve the asymptotic lower bound we have obtained there.
\end{rem}

\subsection{The case $p=3$}

At this stage we are ready to prove \cref{illtheorem} in the case $p=3$. Set $\mathcal{I}_N=[ N-N^{-2\delta},N]$ for $\delta:=\varepsilon+$.

By \cref{mainargument}, if the initial data-to-solution map were $C^3$ it would follow that:
\[1\gtrsim \norm{\langle\cdot\rangle^{s}\widehat{u_{3,N}}(T)}_{L^2_{\xi}(\mathcal{I}_N)},
\]
for any fixed $T$ (as small as necessary).

By \cref{leadingorder}, we may estimate this by the leading order and disregard the lower order terms:
\[1\gtrsim \norm{ N^s \, N^{3(\varepsilon-s)}\, h_{3,N}(\xi)\, e^{-iT|\xi|^{\sigma}}\,N^{\sigma-\alpha}\, \, T}_{L^2_{\xi}(\mathcal{I}_N)}.\]

Using \cref{hpN} and disregarding the lower order terms again, we obtain:
\begin{align*}
1 & \gtrsim N^{3(\varepsilon-s)} \, N^{-4\varepsilon} \, N^{s+\sigma-\alpha}\,T \, |\mathcal{I}_N|^{\frac{1}{2}}\\
& =N^{3(\varepsilon-s)} \, N^{-4\varepsilon} \, N^{s+\sigma-\alpha}\,T \, N^{-\delta} =T\, N^{-2s-\varepsilon+\sigma-\alpha-\delta}.
\end{align*}
This inequality is valid for any fixed $T$ in the lifespan of the solution and any $N>N_0(T)\gtrsim T^{-\frac{1}{\sigma}}$. If the exponent of $N$ were positive, we would reach a contradiction by taking $N\rightarrow\infty$. Therefore, we will have that the data-to-solution map is not $C^3$ whenever: 
\[ -2s-\varepsilon+\sigma-\alpha-\delta>0 \quad \Leftrightarrow\quad -2s+\sigma-\alpha>\varepsilon+\delta.\]
Remember that we have the condition $\varepsilon>\frac{\sigma-1}{2}$ from the hypothesis of \cref{u1prop}. Using this, together with $\delta=\varepsilon+$ we obtain:
\[ -2s+\sigma-\alpha>\sigma-1\quad\Leftrightarrow\quad \frac{1}{2}-\frac{\alpha}{2}=s_c>s,\]
as we wanted to prove.

\subsection{Main argument}

Unfortunately, the argument we used for the case $p=3$ does not provide a sharp result in the case  $p>3$. For that reason, we introduce some scaling symmetry as a way to add an extra parameter that allows us to exploit the time variable.

Consider the rescaled initial data:
\[ f_{\lambda,N}(x):=\lambda^{\frac{-\alpha\beta}{p-1}}\,f_N (\lambda^{-\beta}x),\]
The coefficient $\lambda^{\frac{-\alpha\beta}{p-1}}$ is given by rescaling the problem \cref{perturbed}. This initial data satisfies:
\begin{align}
 \norm{f_{\lambda,N}}_{L^2} & =\lambda^{-\frac{\alpha\beta}{p-1}+\frac{\beta}{2}}\, N^{-s},\nonumber \\
 \norm{f_{\lambda,N}}_{\dot{H}^s} & =\lambda^{-\frac{\alpha\beta}{p-1}+\beta(\frac{1}{2}-s)}.\label{scalednorm}
 \end{align}

We note that
\[ u_{1,\lambda}(t,x)=\lambda^{\frac{-\alpha\beta}{p-1}}\,u_1(\lambda^{-\alpha}t,\lambda^{-\beta}x)\]
solves the rescaled version of \cref{order1}.
This can in turn be used to find
\[ u_{p,\lambda}(t,x)= \lambda^{\frac{-\alpha\beta}{p-1}} u_p (\lambda^{-\alpha} t,\lambda^{-\beta}x),\]
which solves the rescaled version of \cref{orderp}.

We repeat the steps in \cref{mainargument} with these functions. If the initial-data-to-solution map were $C^p$ from $H^s(\RR)$ to $C([0,T],H^s(\RR))$, we would have:
\begin{equation}\label{scaledfirst}
\norm{f_{\lambda,N}}_{H^s}\gtrsim \norm{u_{p,\lambda,N}(T)}_{H^s}\gtrsim \norm{|\cdot|^s \widehat{u_{p,\lambda,N}}(T,\cdot)}_{L^2(\mathcal{I}_{N,\lambda})}.	
\end{equation}

We set $\mathcal{I}_{N,\lambda}=\lambda^{\beta} \mathcal{I}_N$ so that we can use the knowledge we developed about $u_{p,N}$ on this interval. Then
\begin{align}\label{scaledmainargument}
\norm{|\cdot|^s \widehat{u_{p,\lambda,N}}(T,\cdot)}_{L^2(\mathcal{I}_{N,\lambda})} & = \norm{|\cdot|^s \lambda^{-\frac{\alpha\beta}{p-1}+\beta}\widehat{u_{p,N}}(\lambda^{-\alpha}T,\lambda^{\beta}\cdot)}_{L^2(\mathcal{I}_{N,\lambda})}\\
& =\norm{|\cdot|^s \lambda^{-\frac{\alpha\beta}{p-1}+\frac{\beta}{2}-s\beta}\widehat{u_{p,N}}(\lambda^{-\alpha}T,\cdot)}_{L^2(\lambda^{-\beta}\mathcal{I}_{N,\lambda})}\nonumber \\
&=\lambda^{-\frac{\alpha\beta}{p-1}+\frac{\beta}{2}-s\beta}\,\norm{|\cdot|^s\widehat{u_{p,N}}(\lambda^{-\alpha}T,\cdot)}_{L^2(\mathcal{I}_{N})}.\nonumber
\end{align}

We also set $\lambda:=\lambda(N)=N^b$ for some power $b\in\RR$ to be defined later. However, there are some conditions on what this exponent could be. First of all, \cref{futurecond1} gave us a necessary condition for the interval $J_2$ to be nonempty that we now must rescale:
\[ \lambda^{-\alpha}T\gtrsim N^{-\sigma}, \quad \Leftrightarrow\quad T \gtrsim N^{b\alpha-\sigma}.\]
Therefore in order for the estimates that we previously developed to be valid for any small $T$, we require:
\begin{equation}\label{bcond1}
\sigma>b\alpha.
\end{equation}

Similarly, the approximations we developed for $u_1,g(u_1),\ldots$ depended on working with times that satisfied that $O(t N^{-2\varepsilon-1+\sigma})$ was a term of lower order than $O(1)$. See for instance \cref{futurecond2}. This should hold in the whole interval $J_2$, and in particular at the end of that interval $J_2$ after rescaling:
\[ \lambda^{-\alpha}T\,N^{-2\varepsilon-1+\sigma}=o(1),\quad\Leftrightarrow\quad N^{-\alpha b-2\varepsilon-1+\sigma}=o(1).\]
In other words,
\begin{equation}\label{bcond2}
\sigma-2\varepsilon-1<\alpha b.
\end{equation}
Note that this condition replaces the requirement for $\varepsilon$ to be greater than $\frac{\sigma-1}{2}$ that we had in \cref{u1prop} - in fact we expect that it will give us greater flexibility with the range of $\varepsilon$.

\begin{rem} One might wonder whether \cref{bcond1,bcond2} are the only conditions limiting the range of $b$. To find all such conditions one must go back to the approximations we developed for $u_1,g(u_1),\ldots$ in \cref{leadingorder}. However, note that most errors in such approximations depend on the product of powers of $t^{\beta}$ and $N^{\alpha}$, which will be small thanks to \cref{bcond1}. As an example, note that in the asymptotic expansion of $E_{\beta}$ in \cref{Lasymp} there is a term in $t^{-\beta}|\xi|^{-\alpha}$ which, after integration over $J_2$, produced terms similar to $T^{1-\beta} N^{-\alpha}$. When $T$ was fixed, such terms were of lower order in $N$ than the leading term. After scaling, $T$ becomes $\lambda^{-\alpha}\,T=N^{-\alpha b}\,T$ so these terms become $T^{1-\beta}\, N^{-b\alpha(1-\beta)-\alpha}$. Instead, the leading term in \cref{Lasymp} only contributes with $N^{-b\alpha}\,T$. Therefore, for the leading behavior to stay the same, we would need to add the condition:
\[-b\alpha(1-\beta)-\alpha<-b\alpha \quad\Leftrightarrow\quad b\beta<1,\]
which, as expected, agrees with \cref{bcond1}. 
\end{rem}

\begin{rem} In the previous computations, we are assuming that some $H^s$ norms are controlled by $\dot{H}^s$ norms, which is true for the functions we are considering regardless of $s$. For instance, in the case of $f_{\lambda_N}$, we have the following:
\begin{align*}
 \lambda^{\frac{2\alpha\beta}{p-1}}\,\norm{f_{\lambda,N}}_{H^s}^2 & \sim \lambda^{(1-2s)\beta} \, \norm{f_N}_{\dot{H}^s}+\lambda^{(1-2s)\beta} \int_{|\xi|<\lambda^{\beta}} |\widehat{f_N}(\xi)|^2 \,(\lambda^{2s\beta}-|\xi|^{2s})\,d\xi\\
 & = \lambda^{(1-2s)\beta}+ \lambda^{(1-2s)\beta} \int_{|\xi|<\lambda^{\beta}} |\widehat{f_N}(\xi)|^2 \, (\lambda^{2s\beta}-|\xi|^{2s})\,d\xi.
 \end{align*}
 Note that for $N$ large enough
 \[\int_{|\xi|<\lambda^{\beta}} |\widehat{f_N}(\xi)|^2 \, (\lambda^{2s\beta}-|\xi|^{2s})\,d\xi=0,\]
thanks to the fact that the set over which the integration happens will be empty for large enough $N$, given $\xi\in\mathcal{I}_N$, $\lambda=N^b$ and \cref{bcond1}.
The same happens in the bound for $u_{p,\lambda,N}(T)$ used in \cref{scaledfirst}, since we are integrating over $\mathcal{I}_N$.

\end{rem}

We are now ready to finish the proof of \cref{illtheorem}. Suppose that the initial data-to-solution map were $C^p$. Then \cref{scalednorm,scaledfirst,scaledmainargument} would yield:
\[ \lambda^{-\frac{\alpha\beta}{p-1}+\beta(\frac{1}{2}-s)}=\norm{f_{\lambda,N}}_{\dot{H}^s}\gtrsim \lambda^{-\frac{\alpha\beta}{p-1}+\frac{\beta}{2}-s\beta}\,\norm{|\cdot|^s\widehat{u_{p,N}}(\lambda^{-\alpha}T,\cdot)}_{L^2(\mathcal{I}_N)}.\]
As before, we set $\mathcal{I}_N:=[N-N^{-2\delta},N]$ for $\delta=\varepsilon+$. Then the rescaled version of Proposition \cref{leadingorder} gives:
\[1\gtrsim  \norm{N^{p(\varepsilon-s)} \,h_{p,N}(\xi) \, N^{s+\sigma-\alpha} \lambda^{-\alpha}T}_{L^2_{\xi}(\mathcal{I}_N)},\]
after disregarding lower order terms. We now use \cref{hpN}, and disregard the lower order terms again to get a lower bound:
\[1\gtrsim \norm{N^{p(\varepsilon-s)} \, N^{2\varepsilon (1-p)} \, N^{s+\sigma-\alpha}\lambda^{-\alpha}T}_{L^2_{\xi}(\mathcal{I}_N)}.\]
After integration,
\[1\gtrsim N^{p(\varepsilon-s)} \, N^{2\varepsilon (1-p)} \, N^{s+\sigma-\alpha}\, \lambda^{-\alpha}\, T \, |\mathcal{I}_N|^{\frac{1}{2}}.\]
Setting $\lambda:=N^b$ yields:
\[1\gtrsim T \, N^{-\alpha b+p(\varepsilon-s)+2\varepsilon(1-p)+s+\sigma-\alpha-\delta},\]
for any small $T$ and all $N\geq N_0(T)\gtrsim T^{0-}$.
If the exponent of $N$ were positive we would reach a contradiction after taking $N\rightarrow\infty$. Therefore, the data-to-solution map will not be $C^p$ whenever the following condition holds:
\[ -\alpha b +  p(\varepsilon-s)+2\varepsilon(1-p)+s+\sigma-\alpha-\delta>0.\]
After some computations, and using that $\delta=\varepsilon+$ we obtain:
\begin{equation}\label{linprog}
 -\alpha b + (1-p)\varepsilon+\sigma-\alpha>(p-1)s.
\end{equation}
At this stage, we must choose $\varepsilon$ and $b$ that maximize the upper-bound subject to conditions \cref{bcond1,bcond2}. When $p\geq 3$, the optimal choice is to take 
\[b:=\frac{1}{\beta}-, \quad \mbox{and} \quad \varepsilon:=\frac{\sigma-1-\alpha b}{2}+ ,\]
which yields the range
$ s_c >s$.

When $p=2$, the quantity $-\alpha b + (1-p)\varepsilon+\sigma-\alpha$ in \cref{linprog} can be made as large as we want under conditions \cref{bcond1,bcond2}, which shows that the initial data-to-solution map from $H^s(\RR)$ to $C_t([0,T],H^s_x(\RR))$ will not be $C^2$ for any regularity $s$, no matter how large.

\appendix

\section{On the representations of the solutions}
\label{sec: appA}

Consider the linear space-time fractional Schr\"odinger equation \cref{linear} with $g=0$.
By taking the Fourier transform in space and the Laplace transform in time, we obtain
\[i^{\beta} s^{\beta} \widehat{u}(s,\xi) - i^{\beta} s^{\beta -1}  \widehat{f}(\xi) = |\xi |^{\alpha} \widehat{u}(s,\xi),\]
and thus
\[\widehat{u}(s,\xi)= \frac{i^{\beta} s^{\beta -1}}{i^{\beta} s^{\beta} - |\xi |^{\alpha}} \, \widehat{f}(\xi),\]
which we then invert to find:
\[u(t,x)=\int_{\RR} e^{ix\cdot \xi} \, \widehat{f}(\xi) \, \left[ \sum_{k=0}^{\infty} \frac{ t^{\beta\, k} |\xi|^{\alpha\, k} i^{- \beta\, k}}{\Gamma(\beta\, k +1)}\right] \, d\xi , \]
which was already given in \cref{Lsolution}.

 For completeness, we now present the main ideas on how this function formally solves equation \cref{linear} in the case $g=0$. The full details can be found in \cite{kai}.

With $u$ as in \cref{Lsolution}, we formally have 
\[\partial_t u(t,x)= \int_{\RR} e^{ix\cdot \xi} \, \widehat{f}(\xi) \, \left[ \sum_{k=1}^{\infty} \frac{ \beta k  t^{\beta\, k -1} |\xi|^{\alpha\, k} i^{- \beta\, k}}{\Gamma(\beta\, k +1)}\right] \, d\xi ,\]
and therefore,
\[\partial_t^{\beta} u (t,x)= \frac{1}{\Gamma (1-\beta )} \int_{\RR} e^{ix\cdot \xi} \, \widehat{f}(\xi) \, \left[ \sum_{k=1}^{\infty} \frac{ \beta k  |\xi|^{\alpha\, k} i^{- \beta\, k}}{\Gamma(\beta\, k +1)}\right] \, \int_0^t \frac{{\tau}^{\beta\, k -1}}{(t-\tau)^{\beta}} \, d\tau \, d\xi .\]

One can easily check that the $\tau$-integral is essentially a Beta function:
\[\int_0^t \frac{{\tau}^{\beta\, k -1}}{(t-\tau)^{\beta}} \, d\tau= \int_0^1 \frac{t^{\beta\, k -1}\, {\tau}^{\beta\, k -1}}{t^{\beta} \, (1-\tau)^{\beta}} \, t \,d\tau
= t^{\beta (k-1)} \, \frac{\Gamma (\beta k)\, \Gamma (1-\beta) }{\Gamma (\beta k - \beta +1 )} .\]

Consequently, and after using the identity $\Gamma (z+1)=z\,\Gamma (z)$,
\begin{align*}
\partial_t^{\beta} u (t,x) & = \frac{1}{\Gamma (1-\beta )} \int_{\RR} e^{ix\cdot \xi} \, \widehat{f}(\xi) \, \left[ \sum_{k=1}^{\infty} \frac{ \beta k t^{\beta (k-1)}  |\xi|^{\alpha\, k} i^{- \beta\, k} \Gamma (\beta k)\, \Gamma (1-\beta) }{\Gamma(\beta\, k +1) \Gamma (\beta k - \beta +1)}\right] \,  \, d\xi \\
& =\int_{\RR} e^{ix\cdot \xi} \, \widehat{f}(\xi) \, \left[ \sum_{k=1}^{\infty} \frac{ t^{\beta (k-1)}  |\xi|^{\alpha\, k} i^{- \beta\, k} \Gamma (\beta k+1)\,}{\Gamma(\beta\, k +1) \Gamma (\beta k - \beta +1)}\right] \,  \, d\xi \\
& =\int_{\RR} e^{ix\cdot \xi} \, \widehat{f}(\xi) \, \left[ \sum_{k=0}^{\infty} \frac{ t^{\beta k}  |\xi|^{\alpha\, (k+1)} i^{- \beta\, (k+1)}\,}{\Gamma (\beta k +1)}\right] \,  \, d\xi=I.
\end{align*}

On the other hand, we consider the other term involved in equation \cref{linear}, which formally gives
\[(-\Delta_x)^{\alpha /2}\, u=\int_{\RR} e^{ix\cdot \xi} \, \widehat{f}(\xi) \, \left[ \sum_{k=0}^{\infty} \frac{ t^{\beta\, k} |\xi|^{\alpha\, (k+1)} i^{- \beta\, k}}{\Gamma(\beta\, k +1)}\right] \, d\xi=I\! I .\]
Therefore $I$ and $I\! I$ coincide after multiplying the former by the $i^{\beta}$ factor.

Now consider the nonlinear space-time fractional Schr\"odinger equation, as defined in \cref{linear}. Once again, we simply provide a brief exposition, since the full details may be found in  \cite{kai}. A representation for the solution to the nonlinear equation was given in \cref{inteq}, where the Fourier transform of the inhomogeneous part is precisely
\begin{align}\label{h}
h(t,\xi) & =i^{-\beta} \int_0^{t} \widehat{g}(\tau,\xi) \, (t-\tau )^{\beta -1} E_{\beta , \beta }(i^{-\beta} (t-\tau )^{\beta} |\xi |^{\alpha})\, d\tau \\
& =\sum_{k=0}^{\infty}  i^{-\beta (k+1)} |\xi |^{\alpha k}\, (J_{\beta k + \beta }\widehat{g})(t,\xi) ,
\end{align}
where 
\[(J_{\nu} \widehat{g})(t,\xi):=\frac{1}{\Gamma (\nu )} \, \int_0^t \, \widehat{g} (\tau,\xi) (t-\tau )^{\nu -1}\, d\tau \ ,\]
which is based on the definition of $E_{\beta, \beta}$ in \cref{genML}.
One can check that $J_{\nu_1} J_{\nu_2} =J_{\nu_1 + \nu_2}$ for any $\nu_1 ,\nu_2 >0$.

All we need to do now is to show that $h$ is a particular solution of the following Cauchy problem:
\begin{equation}\label{heq}
\left\{ \begin{array}{ll}
i^{\beta} \partial_t^{\beta} \widehat{u} & = |\xi |^{\alpha}\, \widehat{u} + \widehat{g} ,\\
\widehat{u}\mid_{t=0} & = 0 .
\end{array}\right.
\end{equation}

We have the following formal equalities:
\begin{align*}
I & = |\xi |^{\alpha} h(t,\xi )=\sum_{k=0}^{\infty} i^{-\beta (k+1)} |\xi |^{\alpha (k+1)} \, J_{\beta k + \beta } \widehat{g}(t) \\
I\! I & = i^{\beta} \partial_t^{\beta} h(t,\xi )= \sum_{k=0}^{\infty} i^{-\beta k} |\xi |^{\alpha k} \partial_t^{\beta}J_{\beta k + \beta } \widehat{g}(t) .
\end{align*}
Note that $\partial_t^{\beta}=D_t^{\beta}\left(\mbox{id} - \mbox{ev}_0\right)$, where $\mbox{id}$ is the identity, $\mbox{ev}_0f=f(0)$ and 
\[D_t^{\beta}f:=\frac{1}{\Gamma (1-\beta )} \frac{d}{dt} \int_0^t f(s)(t-s)^{-\beta} \, ds=\frac{d}{dt} J_{1-\beta}f .\]
The proof of this fact follows directly from integration by parts. Then one shows the following identity
\[\partial_t^{\beta}J_{\beta}=D_t^{\beta} \left(J_{\beta} - \mbox{ev}_0\,J_{\beta}\right)=D_t^{\beta}J_{\beta}=\frac{d}{dt} J_{1-\beta} J_{\beta}=\frac{d}{dt}J_1 =\mbox{id} ,\]
by the fundamental theorem of calculus, having also used the fact that $\mbox{ev}_0\,J_{\beta}=0$.

With all this in mind, let us show that $h$ in \cref{h} formally satisfies the equation in \cref{heq}:
\begin{align*}
I\! I-I & =\sum_{k=0}^{\infty} i^{-\beta k} |\xi |^{\alpha k} \partial_t^{\beta}J_{\beta} J_{\beta k} \widehat{g} - \sum_{k=0}^{\infty} i^{-\beta (k+1)} |\xi |^{\alpha (k+1)} J_{\beta k +\beta}\widehat{g} \\
& =\sum_{k=0}^{\infty} i^{-\beta k} |\xi |^{\alpha k} J_{\beta k} \widehat{g} - \sum_{k=0}^{\infty} i^{-\beta (k+1)} |\xi |^{\alpha (k+1)} J_{\beta k +\beta}\widehat{g}=J_0 \widehat{g}=\widehat{g} .
\end{align*}
It is also obvious that $h\!\mid_{t=0}=0$.

\section{General proof of \cref{maintheorem}}
\label{sec: appB}

In this section, we provide the proof of \cref{maintheorem} for general values of the parameters involved.
As we did before, we define the following operator based on the integral equation given in \cref{inteq},
\begin{align}\label{fracduhamel2}
\Phi (v)(t,x)  = & \int_{\RR} e^{ix\cdot \xi} \, \widehat{f}(\xi) \, E_{\beta}(|\xi |^{\alpha} t^{\beta} i^{-\beta}) \, d\xi \\
 & + i^{-\beta} \, \int_0^{t} \int_{\RR} \widehat{g}(\tau , \xi) \, (t-\tau )^{\beta -1} E_{\beta , \beta }(i^{-\beta} (t-\tau )^{\beta} |\xi |^{\alpha}) \, e^{ix\cdot \xi}\, d\xi d\tau ,\nonumber
\end{align}
where 
\[\widehat{g}(\tau , \xi)=\int_{\RR} |v(\tau,x)|^{p-1} v(\tau,x) e^{-ix\cdot \xi} \, dx .\]

Now we define 
\begin{align*}
\eta_1 (v) & = \norm{\langle \nabla \rangle^{\delta} v }_{L^{\infty}_x L^2_T} ,\\
\eta_2 (v) & = \norm{\langle \nabla \rangle^s v }_{L^{\infty}_T L^2_x} ,\\
\eta_3 (v) & = \norm{ v }_{L^{2(p-1)}_x L^{\infty}_T} ,
\end{align*}
for some $s$ and some $\delta$ to be chosen later, and let $\Lambda_T:= \max_{j=1,2,3}\, \eta_j$. 
Then consider the space $X_T:=\{v\in C\left([0,T],H^s(\RR )\right) \mid \Lambda_T (v)<\infty \}$. Our goal is to show that for small enough $T$, there exists a ball $B_R \subset X_T$ such that $\Phi : B_R \longrightarrow B_R$ is a contraction, and then apply the contraction mapping principle.

Let us start by taking the first norm of \cref{fracduhamel2}, where we will use \cref{thsmooth} and \cref{series}.
\begin{align}\label{geta1}
\eta_1 (\Phi (v)) \lesssim & \,(1+ T^{\frac{1}{2}})\, (1+ T^{\frac{\gamma - \gamma '}{\sigma}})\, \norm{\langle \nabla \rangle^{\delta-\gamma'} f}_{L^2_x} \\
& + (1+T^{\frac{\tilde{\gamma} - \tilde{\gamma}'}{\sigma}} +T^{\frac{1}{2}})\, \norm{\langle \nabla \rangle^{\delta-\tilde{\gamma}'} (|v|^{p-1} v)}_{L^1_T L^2_x} .\nonumber
\end{align}

Now we look at the second norm of \cref{fracduhamel}. \Cref{L21,L22,L23} combined with \cref{series} yield
\begin{equation}\label{geta2}
\eta_2 (\Phi (v))\lesssim \norm{\langle \nabla \rangle^s f}_{L^2_x} + T^{\beta - \frac{1}{2}}\, (1+ T^{1-\beta}) \, \norm{ \langle\nabla\rangle^{s+\sigma -\alpha} (|v|^{p-1} v)}_{L^2_{T,x}} .
\end{equation}
In order to be able to eventually control the nonlinear part with $\eta_1$, we require $s+\sigma-\alpha\leq \delta$, as well as $s\geq \delta - \gamma'$ to control the linear part too. These two conditions on $s,\delta$ are perfectly compatible, and one may easily check that are equivalent to $\alpha>\frac{\sigma+1}{2}$ as stated in the hypothesis of \cref{maintheorem}, see \cref{parameterconditions}.

Finally, we take the third norm, which we control thanks to \cref{maxth} and \cref{series}.
\begin{equation}\label{geta3}
\eta_3 (\Phi (v))\lesssim \norm{\langle \nabla \rangle^{s} f}_{L^2_x}+ T^{\beta - \frac{1}{2}}\, \norm{ \langle\nabla\rangle^{s+\sigma -\alpha} (|v|^{p-1} v)}_{L^2_{T,x}} ,
\end{equation}
for $s\geq\frac{1}{2}-\frac{1}{2(p-1)}$ and $p\geq 3$.

Therefore, we can simultaneously estimate \cref{geta1,geta2,geta3} by controlling the quantity $\norm{ \langle\nabla\rangle^{s+\sigma -\alpha} (|v|^{p-1} v)}_{L^2_{T,x}}$ in terms of $\eta_1 , \eta_2$ and $\eta_3$. Using \cref{AvoidChainRule}, we have
\begin{equation}\label{usefulforcont}
\norm{ \langle\nabla\rangle^{s+\sigma -\alpha} (|v|^{p-1} v)}_{L^2_{T,x}}\lesssim \eta_1 (v) \, \eta_3 (v)^{p-1} \, .
\end{equation}

By assuming that $T\leq 1$, \cref{geta1,geta2,geta3} can be rewritten as
\[\Lambda_T (\Phi (v))\lesssim \norm{f}_{H^s_x}+ T^{\beta - \frac{1}{2}} \, \eta_1 (v) \, \eta_3 (v)^{p-1}\lesssim \norm{f}_{H^s_x}+ T^{\beta - \frac{1}{2}} \, \Lambda (v)^p . \]

Pick $B_R = \{ v\in X_T \mid \Lambda_T (v) < R \}$ and $f$ such that $\norm{f}_{H^s_x}\lesssim \frac{R}{2}$. Then we have that $\Phi : B_R \longrightarrow B_R$ as long as $C\, T^{\beta - \frac{1}{2}}\, R^p < \frac{R}{2}$, which happens for $T$ small enough. We still must prove that $\Phi (v)$ is actually in $C\left([0,T],H^s(\RR )\right)$, but we do that in \cref{cont} below.

Now we prove that $\Phi$ is a contraction. By using the same ideas as in \cref{geta1,geta2,geta3} together with $T\leq 1$ we quickly find
\[\Lambda_T (\Phi (v)- \Phi (u))\lesssim T^{\beta - \frac{1}{2}}\, \norm{ \langle\nabla\rangle^{s+\sigma -\alpha} (|v|^{p-1} v - |u|^{p-1} u)}_{L^2_{T,x}}.\]
Now we adapt an idea found in the proof of Theorem 1.2 in \cite{HS}, which is based on the fundamental theorem of calculus:
\begin{align*}
\norm{|v|^{p-1} v - |u|^{p-1} u}_{L^2_T H^{s+\sigma-\alpha}_x} & \leq \norm{\int_0^1 p |v + \lambda (u-v) |^{p-1} (u-v) \, d\lambda }_{L^2_T H^{s+\sigma-\alpha}_x} \\
& \leq p \int_0^1 \norm{ |v + \lambda (u-v) |^{p-1} (u-v)}_{L^2_T H^{s+\sigma-\alpha}_x} \, d\lambda .
\end{align*}
Suppose $p=2k+1$. By \cref{AvoidChainRule} we have
\begin{align*}
\left\lVert \right.&\left.\langle\nabla\rangle^{s+\sigma -\alpha} \left(| v + \lambda (u-v) |^{p-1} (u-v)\right)\right\rVert_{L^2_{T,x}} \\
   =& \norm{ \langle\nabla\rangle^{s+\sigma -\alpha} \left([v + \lambda (u-v) ]^{k} [\bar{v} + \lambda (\bar{u}-\bar{v}) ]^{k} (u-v)\right)}_{L^2_{T,x}}\\
 \lesssim & \norm{ \langle\nabla\rangle^{s+\sigma -\alpha} \left(v+\lambda (u-v) \right) [v + \lambda (u-v) ]^{k-1} [\bar{v} + \lambda (\bar{u}-\bar{v}) ]^{k} (u-v)}_{L^2_{T,x}} \\
 & +\norm{ \langle\nabla\rangle^{s+\sigma -\alpha} \left(\bar{v} + \lambda (\bar{u}-\bar{v}) \right) [v + \lambda (u-v) ]^{k} [\bar{v} + \lambda (\bar{u}-\bar{v}) ]^{k-1} (u-v)}_{L^2_{T,x}} \\
& +\norm{ \langle\nabla\rangle^{s+\sigma -\alpha} (u-v) [v + \lambda (u-v) ]^{k} [\bar{v} + \lambda (\bar{u}-\bar{v}) ]^{k} }_{L^2_{T,x}} = A+ B +C .
\end{align*}
A bound for $C$ follows from the H\"older inequality.
\begin{align*}
C & \leq \norm{ \langle\nabla\rangle^{s+\sigma -\alpha} (u-v)}_{L^{\infty}_x L^2_T} \norm{[v + \lambda (u-v) ]^{k} [\bar{v} + \lambda (\bar{u}-\bar{v}) ]^{k}}_{L^2_x L^{\infty}_T}\\
& =\eta_1 (u-v) \norm{|v + \lambda (u-v)|^{p-1}}_{L^2_x L^{\infty}_T}= \eta_1 (u-v)\norm{v + \lambda (u-v)}_{L^{2(p-1)}_x L^{\infty}_T}^{p-1}\\
& =\eta_1 (u-v)  \, \eta_3 (v + \lambda (u-v))^{p-1} \leq \eta_1 (u-v) \, [\eta_3 (v)^{p-1} + \eta_3 (u)^{p-1} ]\\
&\leq 2 R^{p-1}\,  \eta_1 (u-v) .
\end{align*}

Now let's deal with $A$, which also follows from the H\"older inequality.
\begin{multline*}
A  =\norm{ \langle\nabla\rangle^{s+\sigma -\alpha} \left(v+\lambda (u-v) \right) [v + \lambda (u-v) ]^{k-1} [\bar{v} + \lambda (\bar{u}-\bar{v}) ]^{k} (u-v)}_{L^2_{T,x}} \\
  \leq \norm{u-v}_{L^{2(p-1)}_x L^{\infty}_T} \norm{ \langle\nabla\rangle^{s+\sigma -\alpha} \left(v+\lambda (u-v) \right)}_{L^{\infty}_x L^2_T} \\
 \norm{[v + \lambda (u-v) ]^{k-1} [\bar{v} + \lambda (\bar{u}-\bar{v}) ]^{k}}_{L^r_x L^{\infty}_T} ,
\end{multline*}
where $\frac{1}{2}=\frac{1}{2(p-1)} + \frac{1}{r}$.  $B$ is done in the same way.

Then one uses the pointwise bound $[v + \lambda (u-v) ]^{k-1}\lesssim |v|^{k-1} + |u|^{k-1}$, 
\begin{align*}
A & \leq \eta_3 (u-v)\,  \norm{ \lambda\, \langle\nabla\rangle^{s+\sigma -\alpha} u + (1-\lambda)\, \langle\nabla\rangle^{s+\sigma -\alpha} v }_{L^{\infty}_x L^2_T}\norm{|v|^{2k-1} + |u|^{2k-1}}_{L^r_x L^{\infty}_T} \\
 & \leq \eta_3 (u-v)\, (\eta_1 (u) + \eta_1 (v)) \, ( \eta_3 (u)^{p-2} + \eta_3 (v)^{p-2})\leq 4R^{p-1}\eta_3 (u-v), 
\end{align*}
after checking that $(2k-1)r=2(p-1)$.

Putting everything together, we obtain
\[\Lambda_T (\Phi (v)- \Phi (u))\lesssim T^{\beta - \frac{1}{2}}\,R^{p-1} \,\Lambda_T (v-u),\]
so that by making $T$ smaller (if necessary), we can arrange $T^{\beta - \frac{1}{2}}\,R^{p-1}<1$. Therefore $\Phi$ is a contraction on the ball $B_R = \{ v \mid \Lambda_T (v) < R\sim \norm{f}_{H^s_x} \}$ with $s=\frac{1}{2}-\frac{1}{2(p-1)}$.

Finally, we give the proof that $\Phi (v)$ is continuous for completeness.

\begin{lemma}\label{cont} $\Phi (v)\in C\left([0,T],H^s(\RR )\right)$ whenever $v\in X_T$.
\end{lemma}
\begin{proof}
Let $\mbox{lin}_t v$ be the linear part of $\Phi (v)$. It is then easy to show that $\mbox{lin}_t v\in C\left([0,T],H^s(\RR )\right)$. Firstly, we use the Plancherel theorem to write
\[\norm{\mbox{lin}_t v}_{H^s_x}=\norm{|\widehat{f} |^2 \, \langle \xi\rangle^{2s}\, |E_{\beta}(i^{-\beta}t^{\beta}|\xi |^{\alpha})|^2}_{L^1_{\xi}}^{\frac{1}{2}} \ .\]
Then we use the Dominated convergence theorem, together with the fact that 
\[|\widehat{f} |^2 \, \langle \xi\rangle^{2s}\, |E_{\beta}(i^{-\beta}t^{\beta}|\xi |^{\alpha})|^2\lesssim_M  |\widehat{f} |^2 \, \langle \xi\rangle^{2s}\in L^1_{\xi}\]
 uniformly in $t$, and also the fact that $E_{\beta}(i^{-\beta}t^{\beta}|\xi |^{\alpha})$ is continuous in $t$.

Therefore, we only need to prove that $\mbox{non}_t v$, the nonlinear part of $\Phi (v)$, also lives in $C\left([0,T],H^s(\RR )\right)$. Suppose $0\leq t_2 <t_1\leq T$, and consider
\begin{align*}
\left\lVert \mbox{non}_{t_1} v\right. &-\left.\mbox{non}_{t_2} v\right\rVert^2_{H^s_x} \\
 = & \int_{\RR} \Big | \int_0^{t_1} \widehat{g}(t_1 - \tau ,\xi) \, \tau^{\beta-1} \, E_{\beta, \beta}(i^{-\beta}\tau^{\beta}|\xi |^{\alpha}) \, d\tau \\
& - \int_0^{t_2} \widehat{g}(t_2 - \tau ,\xi) \, \tau^{\beta-1} \, E_{\beta, \beta}(i^{-\beta}\tau^{\beta}|\xi |^{\alpha}) \, d\tau \Big |^2 \langle\xi\rangle^{2s}\,d\xi \\
 = & \int_{\RR} \Big | \int_{t_2}^{t_1} \widehat{g}(t_1 - \tau ,\xi) \, \tau^{\beta-1} \, E_{\beta, \beta}(i^{-\beta}\tau^{\beta}|\xi |^{\alpha}) \, d\tau \\
& + \int_0^{t_2} \left[ \widehat{g}(t_1 - \tau ,\xi)- \widehat{g}(t_2 - \tau ,\xi)\right] \, \tau^{\beta-1} \, E_{\beta, \beta}(i^{-\beta}\tau^{\beta}|\xi |^{\alpha}) \, d\tau \Big |^2 \langle\xi\rangle^{2s}\,d\xi \\
 \lesssim & I + I\! I ,
\end{align*}
where 
\begin{align*}
I & = \int_{\RR} \Big |  \int_{t_2}^{t_1} \widehat{g}(t_1 - \tau ,\xi) \, \tau^{\beta-1} \, E_{\beta, \beta}(i^{-\beta}\tau^{\beta}|\xi |^{\alpha}) \, d\tau \Big |^2\langle\xi\rangle^{2s}\,d\xi ,\\
I\! I & = \int_{\RR} \Big | \int_0^{t_2} \left[ \widehat{g}(t_1 - \tau ,\xi)- \widehat{g}(t_2 - \tau ,\xi)\right] \, \tau^{\beta-1} \, E_{\beta, \beta}(i^{-\beta}\tau^{\beta}|\xi |^{\alpha}) \, d\tau \Big |^2 \langle\xi\rangle^{2s}\,d\xi .
\end{align*}
We first focus on $I$. By \cref{series}, the multiplier $\tau^{\beta-1} \, E_{\beta, \beta}(i^{-\beta}\tau^{\beta}|\xi |^{\alpha})$ can be controlled by those associated to $|\nabla |^{\sigma-\alpha} e^{-i\tau |\nabla |^{\sigma}}$, $\tilde{T}_{\tau}$, $\tilde{S}_{\tau}$ and $\tilde{U}_{\tau}$, and so we should consider each case separately. We treat one as an example. As $t_2\rightarrow t_1$,
\begin{align*}
 \int_{\RR} \Big |  \int_{t_2}^{t_1} \widehat{g}(t_1 - \tau ,\xi) \, & |\xi |^{\sigma -\alpha} \, e^{-i\tau |\xi |^{\sigma}} \, d\tau \Big |^2\langle\xi\rangle^{2s}\,d\xi \\
 & \leq \int_{\RR} \left(  \int_{t_2}^{t_1} |\widehat{g}(t_1 - \tau ,\xi)|\, d\tau \right)^2 \langle\xi\rangle^{2(s+\sigma-\alpha )}\,d\xi\\
& \leq \norm{g(t_1 - \cdot )}_{L^2 ([t_2 ,t_1], H^{s+\sigma-\alpha}_x)}^2 \, |t_1-t_2|\\
&\leq \norm{\langle\nabla\rangle^{s+\sigma-\alpha} g}_{L^2_{T,x}}^2 \, |t_1 - t_2| \rightarrow 0 
\end{align*}
 where we used the Cauchy-Schwartz inequality and the fact that $\norm{\langle\nabla\rangle^{s+\sigma-\alpha} g}_{L^2_{T,x}}$ is finite as long as $v\in X_T$, which was shown in \cref{usefulforcont}. The other cases are treated analogously.

Now we focus on $I\! I$. By \cref{series}, one can control this quantity by the multipliers associated to $|\nabla |^{\sigma-\alpha} e^{-i\tau |\nabla |^{\sigma}}$, $\tilde{T}_{\tau}$, $\tilde{S}_{\tau}$ and $\tilde{U}_{\tau}$. Once again, we treat one case as an example:
\begin{align*}
& \int_{\RR} \Big | \int_0^{t_2} \left[ \widehat{g}(t_1 - \tau ,\xi)- \widehat{g}(t_2 - \tau ,\xi)\right] \, |\xi |^{\sigma-\alpha} \, e^{-i\tau |\xi |^{\sigma}} \, d\tau \Big |^2 \langle\xi\rangle^{2s}\,d\xi  \\
& \leq \int_{\RR} \left( \int_0^{t_2} \Big | \widehat{g}(t_1 - \tau ,\xi)- \widehat{g}(t_2 - \tau ,\xi)\Big | \, d\tau \right)^2 \langle\xi\rangle^{2(s+\sigma-\alpha)}\,d\xi \\
& \leq t_2 \, \int_{\RR}  \int_0^{t_2} | \widehat{g}(t_1 - \tau ,\xi)- \widehat{g}(t_2 - \tau ,\xi) |^2 \, \langle\xi\rangle^{2(s+\sigma-\alpha)}\, d\tau \,  d\xi \\
& =t_2 \, \int_{\RR}  \int_0^T \chi_{\{0\leq \tau \leq t_2\}}\, | \widehat{g}(t_1-t_2 + \tau ,\xi)- \widehat{g}(\tau ,\xi) |^2 \, \langle\xi\rangle^{2(s+\sigma-\alpha)}\, d\tau \,  d\xi .
\end{align*}
Now we take $t_2\rightarrow t_1$ and use the dominated convergence theorem, since 
\begin{multline*}
\langle\xi\rangle^{2(s+\sigma-\alpha)} \, \int_0^T \chi_{\{0\leq \tau \leq t_2\}}\, | \widehat{g}(t_1-t_2 + \tau ,\xi)- \widehat{g}(\tau ,\xi) |^2 \, d\tau \\
\leq 2 \langle\xi\rangle^{2(s+\sigma-\alpha)} \, \int_0^T |\widehat{g}(\tau,\xi)|^2 \, d\tau \in L^1_{\xi}
\end{multline*}
independently of $t_2$, as was shown in \cref{usefulforcont}. Of course, we also need to show that 
\[\int_0^T \chi_{\{0\leq \tau \leq t_2\}}\, | \widehat{g}(t_1-t_2 + \tau ,\xi)- \widehat{g}(\tau ,\xi) |^2 \, d\tau\]
 is continuous on $t_2$. But this follows from the translation continuity of $L^p$ norms, together with the fact that $\widehat{g}(\cdot ,\xi) \in L^2([0,T])$ for a.e. $\xi$.
\end{proof}

\section*{Acknowledgments}
The author would like to thank his advisor, Gigliola Staffilani, for all her advice and encouragement. He would also like to thank Yannick Sire, David Jerison and Ethan Jaffe for several useful discussions and suggestions, as well as the referees for their useful comments.
This work was supported by NSF grants DMS-1500771, DMS-1462401, DMS-1764403 and DMS-1362509.

\bibliographystyle{amsplain}
\bibliography{references}
\end{document}